\documentclass[11pt,a4paper,reqno]{amsart}
\usepackage[T1]{fontenc}
\usepackage{lmodern,amsmath,amssymb,amsthm,mathtools,microtype}
\usepackage[margin=27mm]{geometry}
\usepackage{needspace,etoolbox,enumitem,xcolor}
\usepackage[colorlinks=true,linkcolor=blue!50!black,citecolor=blue!50!black,urlcolor=blue!50!black]{hyperref}
\pretocmd{\section}{\Needspace{10\baselineskip}}{}{}
\pretocmd{\subsection}{\Needspace{5\baselineskip}}{}{}
\newtheorem{theorem}{Theorem}[section]
\newtheorem{lemma}[theorem]{Lemma}
\newtheorem{proposition}[theorem]{Proposition}
\newtheorem{corollary}[theorem]{Corollary}
\theoremstyle{remark}
\newtheorem{remark}[theorem]{Remark}
\newcommand{\R}{\mathbb R}
\newcommand{\C}{\mathbb C}
\newcommand{\N}{\mathbb N}
\newcommand{\SO}{\mathrm{SO}}
\newcommand{\Gr}{\mathrm{Gr}}
\newcommand{\dd}{\,\mathrm d}
\newcommand{\Lip}{\operatorname{Lip}}
\newcommand{\supp}{\operatorname{supp}}
\newcommand{\dimH}{\dim_{\mathrm H}}
\newcommand{\HH}{\mathcal H}
\newcommand{\LL}{\mathcal L}
\newcommand{\Pp}{\mathbb P}
\newcommand{\E}{\mathbb E}
\newcommand{\op}{\mathrm{op}}
\newcommand{\Frob}{\mathrm F}
\newcommand{\ind}{\mathbf 1}
\newcommand{\diam}{\operatorname{diam}}

\newcommand{\SSC}{\hyperref[ass:SSC]{\textup{(SSC)}}}
\newcommand{\DRC}{\hyperref[ass:DRC]{\textup{(DRC)}}}

\numberwithin{equation}{section}
\setlist[enumerate]{itemsep=4pt,topsep=6pt}
\title[Projections and fibres of self-similar sets and measures]{Absolute continuity of projections and dimension conservation for self-similar sets and measures}
\author{Xiong Jin}
\address{Department of Mathematics, The University of Manchester,
Oxford Road, Manchester M13 9PL, United Kingdom}
\email{xiong.jin@manchester.ac.uk}
\author{Tuomas Sahlsten}
\address{Department of Mathematics and Statistics, University of Helsinki,
P.O. Box 68, FI-00014 University of Helsinki, Finland}
\email{tuomas.sahlsten@helsinki.fi}
\hypersetup{pdftitle={Absolute continuity of projections and dimension conservation for self-similar sets and measures},
pdfauthor={Xiong Jin and Tuomas Sahlsten}}

\thanks{T.S. acknowledges support from the Research Council of Finland's Academy Research Fellowship \emph{``Quantum chaos of large and many body systems''}, grant Nos. 347365, 353738. Part of this research was carried out during T.S.'s visit to the programme at Institut Mittag-Leffler \emph{Interactions between Fractal Geometry, Harmonic Analysis and Dynamical Systems} (Autumn 2026).}

\begin{document}
\begin{abstract}
Let $A\subset\R^d$, $d\geq3$, be the attractor of a self-similar IFS whose rotations generate a dense subgroup of $\SO(d)$. For every $1\leq k<\dimH A$ and $V\in\Gr(d,k)$, we prove that $\pi_VA$ has positive $k$-dimensional Lebesgue measure uniformly in $V$, and $\dimH(A\cap\pi_V^{-1}\{y\})=\dimH A-k$ for Lebesgue-almost every $y\in\pi_VA$. Moreover, if $\mu$ is a self-similar measure for such an IFS with $I_t(\mu)<\infty$ for some $k<t<d$, then every $\mu_V=(\pi_V)_*\mu$ is equivalent to the $k$-dimensional Lebesgue measure restricted to $\pi_VA$, with density $f_V$ satisfying $\sup_V\int_{\{f_V>M\}}f_V\,d\mathcal L_V^k\lesssim e^{-c(\log M)^{1/3}}$. Under strong separation, the conditional measures on the fibres are exact dimensional of dimension $\dimH\mu-k$. The proof uses self-similarity to iterate $L^1$ norms of differences of the Gaussian-smoothing of $\mu_V$ between scales. Unlike $L^q$, $q>1$, the $L^1$ norm does not grow under rescaling, allowing Varj\'u's $L^2$ weak gap theorem under dense support condition to be used, and avoiding to use the much stronger uniform spectral gap assumption.
\end{abstract}
\maketitle

\section{Introduction}\label{sec:intro}

\subsection{Background}

The Marstrand-Mattila projection theorem \cite{Marstrand54,Mattila75} states that for any $1\leq k<d$ and $A\subset\R^d$ the orthogonal projection $\pi_V A$ on $V\in\Gr(d,k)$ satisfies
\[
    \dimH \pi_V A=\min\{k,\dimH A\}
    \quad\text{for almost every }V\in\Gr(d,k).
\]
Moreover, when $\dimH A>k$, the $k$-dimensional Lebesgue measure $\mathcal L_V^k (\pi_V A) > 0$ for almost every $V$. Here $\Gr(d,k)$ is the Grassmannian of $k$-dimensional linear subspaces of $\R^d$, $\pi_V$ the orthogonal projection onto $V$ and ``almost every'' refers to the rotation-invariant probability measure $\sigma_k$ on $\Gr(d,k)$. The corresponding result for Borel measures $\mu$ on $\R^d$ states that finite $t$-energy $I_t(\mu)$ for $t > k$ implies $L^2$ densities for almost every $(\pi_V)_* \mu$.

There has been considerable interest over the past decades in understanding and reducing the exceptional set in such projection theorems. One direction is to ask whether projection theorems continue to hold when $V$ is restricted to a smaller family in $\Gr(d,k)$ \cite{FasslerOrponen2014,Orponen2015,OrponenVenieri2020,Harris2023,KaenmakiOrponenVenieri2025,GanGuoWang2024,GanGuthMaldague2024,GanGuoGuthHarrisMaldagueWang2026}. These questions are closely related to recent developments such as the proof of the Furstenberg set conjecture in the plane \cite{OrponenShmerkin2023,ShmerkinWang2025,OrponenShmerkin2026,RenWang2023}, the Kakeya conjecture in $\R^3$ \cite{WangZahl2025}, an effective version of the Oppenheim conjecture \cite{LindenstraussMohammadiWangYang2023}, and recent applications to fractal uncertainty principles and spectral gaps \cite{CohenDyatlov26}. Another approach is to exploit additional structure of the set $A$ or measure $\mu$ itself, such as Ahlfors-David regularity \cite{Wu25} or dynamical structure such as self-similarity \cite{HochmanShmerkin12,FalconerJin14,ShmerkinSolomyak16,Rapaport20,
AlgomShmerkin25}. For self-similar sets and measures, stronger projection theorems are known under suitable assumptions on the rotations of the IFS. Peres and Shmerkin~\cite{PeresShmerkin09} proved that
$\dimH \pi_V A=\min\{1,\dimH A\}$ for every $V\in\Gr(2,1)$ when $A\subset\R^2$ is self-similar and the IFS contains an irrational rotation. Later, in $\mathbb{R}^d$ for $d\ge 2$, Hochman and Shmerkin~\cite{HochmanShmerkin12} considered self-similar measures $\mu$ under the strong separation condition and dense rotation conditon, they proved
$\dimH\pi_V\mu=\min\{k,\dimH\mu\}$ for every $V\in\Gr(d,k)$. Further results in this direction were obtained by Falconer and the first author~\cite{FalconerJin14}, and more recently by Algom and Shmerkin~\cite{AlgomShmerkin25}.

Absolute continuity and the dimensions of the fibres are not yet fully understood and the entropy methods such as in \cite{HochmanShmerkin12} do not apply. In this setting Shmerkin and Solomyak
\cite[Theorem~B]{ShmerkinSolomyak16} proved that self-similar measures of dimension greater than one, generated by planar homogeneous IFSs
(i.e. with a common linear part) satisfying strong separation, have
absolutely continuous projections outside a set of directions of
Hausdorff dimension zero. For general sets in $\R^3$, Gan, Guo, Guth, Harris, Maldague and Wang \cite{GanGuoGuthHarrisMaldagueWang2026} proved that if $\dimH A>2$, then  $\mathcal L_V^2(\pi_VA)>0$ for almost every $V$ in a non-degenerate one-parameter family of planes.
The fibre question is naturally expressed through \textit{dimension
conservation} (introduced by Furstenberg, see e.g. \cite{Furstenberg08}), which relates the dimension of a set to those of its
projection and fibres, and has a measure-theoretic counterpart
involving conditional measures. Rapaport \cite{Rapaport17} showed, in $\mathbb{R}^2$, that singular projections and failure of dimension conservation can occur even with dense rotations. In subsequent work \cite{Rapaport20}, Rapaport established absolute continuity and dimension conservation in every direction for a parametrised family of homogeneous self-similar measures in $\mathbb{R}^2$, outside an exceptional set of parameters of Hausdorff dimension zero.

A different approach to absolute continuity uses random walks on groups. Here the work of Lindenstrauss and Varj\'u \cite{LindenstraussVarju16} used
spectral-gap estimates for transfer operators associated to the random walk to prove smoothness of self-similar
measures in $\mathbb{R}^3$ in a certain parameter range. See also \cite{Kittle2024,KittleKogler24}. Moreover, in $\mathbb{R}^d$ for $d\ge 3$, a recent work of Algom, Rodriguez-Hertz and Wang
\cite{AlgomRodriguezHertzWang26} implies criteria for absolute
continuity and Sobolev or Besov regularity of projections in terms of
quantitative mixing of the rotation action. Obtaining a spectral gap
from dense generation alone is a difficult open problem. Bourgain and
Gamburd \cite{BourgainGamburd08,BourgainGamburd12} proved spectral
gaps for dense subgroups of $\mathrm{SU}(d)$ generated by algebraic
elements, and Benoist and de Saxc\'e \cite{BenoistdeSaxce16}
extended this to arbitrary connected compact simple Lie groups.
No corresponding theorem is known for general dense generators making it challenging to tackle also the absolute continuity questions.
We give more detailed comparisons with previous work in
Section~\ref{sec:comparison}.

In this paper we introduce an approach for absolute continuity of projections and dimension conservation for fibres that does not require any spectral gap or quantitative mixing condition from the rotation group. We prove a theorem that applies to every projection
of a self-similar set in $\R^d$, $d\ge3$, whose rotations generate
a dense subgroup of $\SO(d)$. If the attractor has dimension $s>k$,
then its $k$-dimensional projections have a uniform positive
lower bound on their volumes, and every projection has fibres
of dimension $s-k$ at Lebesgue-almost every point of its image.
This theorem allows inhomogeneous contraction ratios and arbitrary overlaps. This follows from a quantitative absolute-continuity
theorem for projections of self-similar measures. We obtain this theorem by combining Varj\'u's $L^2$ estimate for transfer operators for the random walk on rotations with a Gaussian smoothing of martingale difference type sequences
in $L^1$. This bounding is inspired by \cite{LindenstraussVarju16,Sahlsten25} but is completely new. We hope this approach will be interesting approach to future questions about absolute continuity of fractal measures, see Section \ref{sec:prospects} for discussion.

\subsection{Main results}

We begin with the theorem for self-similar sets, and then state the
measure theorem used in its proof. The latter also implies quantitative
bounds for the projected densities. Let $d\ge3$, and let $A = \bigcup_{i = 1}^m S_i A$ be the attractor of a finite
iterated function system (IFS) of similarities
\begin{equation}\label{eq:ifs}
 S_i(x)=r_iO_i x+a_i,\qquad
 0<r_i<1,\quad O_i\in\SO(d),\quad 1\le i\le m.
\end{equation}
We say the IFS satsfies
\begin{description}[leftmargin=4em,labelwidth=3.5em,labelsep=.5em]
\item[(DRC)]\phantomsection\label{ass:DRC}
\textbf{Dense rotation condition} if the rotations generate a dense subgroup of the full rotation group:
\[
 \overline{\langle O_1,\ldots,O_m\rangle}=\SO(d).
\]
\item[(SSC)]\phantomsection\label{ass:SSC}
\textbf{Strong separation condition} if $S_iA$, $1\le i\le m$, are pairwise disjoint.
\end{description}
The system is \emph{equicontractive} if all the ratios $r_i$
equal a common value $\rho$. We write $\LL_V^k$ for Lebesgue
measure on $V$. For a set $A$ with $s=\dimH A>k$, we say the set $A$ is \textit{dimension conserving} for projection $\pi_V$ if
\begin{equation}\label{eq:dc-set}
 \dimH\{y\in\pi_VA:\dimH(A\cap\pi_V^{-1}\{y\})\ge s-k\}=k.
\end{equation}

Our first theorem requires only \DRC: the contraction ratios may
differ and we can have any kind of overlaps:

\begin{theorem}\label{thm:sets}
Let $A\subset\R^d$, $d\ge3$, be the attractor of the
IFS \eqref{eq:ifs} satisfying \DRC.
Write $s=\dimH A$, and let $1\le k<d$ be an integer
with $s>k$. Then we have
\[
 \inf_{V\in\Gr(d,k)}\LL_V^k(\pi_VA)>0.
\]
Moreover, for every $V\in\Gr(d,k)$,
\[
 \dimH(A\cap\pi_V^{-1}\{y\})=s-k
 \quad\text{for }\LL_V^k\text{-almost every }y\in\pi_VA.
\]
In particular, every such projection is dimension conserving
in the set sense of \eqref{eq:dc-set}.
\end{theorem}

Theorem~\ref{thm:sets} is follows from the following absolute continuity theorem for self-similar measures $\mu=\sum_{i=1}^m p_i(S_i)_*\mu$, $p_i>0$, $\sum_{i=1}^m p_i=1$. Let us introduce some notation first. For $V\in\Gr(d,k)$, denote the projected measure $\mu_V:=(\pi_V)_*\mu$ and let $\mu_{V,y}$ be the conditional measure on the fibre based at $y$ (see Section~\ref{subsec:disintegration}). We say that the projection $\pi_V$ is \textit{dimension conserving} (for measures) if
\begin{equation}\label{eq:dc-measure}
 \dimH\mu=\dimH\mu_V+\dimH\mu_{V,y}
 \quad\text{for }\mu_V\text{-almost every }y.
\end{equation}
We will use Gaussian smoothing to measure the regularity of $\mu_V$.
Let $\Gamma_h^V$ be the normalized Gaussian kernel on $V$,
with variance parameter $h^2$:
$\Gamma_h^V(y)
 =(2\pi h^2)^{-k/2}e^{-|y|^2/(2h^2)}$, $y\in V$, $h>0.$ For $0<t<d$, the $t$-\textit{energy} of a Borel measure
$\nu$ is $I_t(\nu):=\iint |x-z|^{-t}\dd\nu(x)\dd\nu(z)$.

\begin{theorem}\label{thm:energy}
Let $d\ge3$ and $1\le k<d$ be integers. Consider the IFS
\eqref{eq:ifs}, with attractor $A$, and assume \DRC.
Let $\mu$ be a self-similar measure associated to the IFS with positive weights. Suppose $I_{t_0}(\mu)<\infty$ for some $k<t_0<d$.
Then every $\mu_V$ has a density $f_V$, and there are constants
$C,c>0$ such that, for $j\ge1$ and $M\ge1$,
\begin{align}
 \sup_{V\in\Gr(d,k)}\|f_V-\mu_V*\Gamma_{2^{-j}}^V\|_1
 &\le C e^{-c j^{1/3}},\label{eq:main-rate}\\
 \sup_{V\in\Gr(d,k)}\int_{\{f_V>M\}}f_V\dd\LL_V^k
 &\le C e^{-c(\log M)^{1/3}}.\label{eq:main-tail}
\end{align}
Furthermore,
\[
 \mu_V\sim\LL_V^k|_{\pi_V A} \text{ and }
 \inf_{V\in\Gr(d,k)}\LL_V^k(\pi_V A)>0.
\]
If the IFS also satisfies \SSC, then, for every $V\in\Gr(d,k)$,
the conditional measures $\mu_{V,y}$ are exact dimensional of
dimension
\[
 \frac{h(p)}{\chi(p)}-k
 =\dimH\mu-k
\]
for $\mu_V$-almost every $y$, and hence for
$\LL_V^k$-almost every $y\in\pi_VA$, where
\[
 h(p)=-\sum_i p_i\log p_i,\qquad
 \chi(p)=-\sum_i p_i\log r_i.
\]
In this case, every projection is dimension conserving in the
measure sense of \eqref{eq:dc-measure}.
\end{theorem}

We do not require any arithmetic condition on the contraction ratios. For the proof, we first prove the density estimates for equalcontraction IFS in
Section~\ref{sec:smoothing}, then extend them to arbitrary contractions in
Section~\ref{sec:unequal}. The conditional-measure and equivalence
arguments in Sections~\ref{sec:conditional} and~\ref{sec:equivalence}
complete the proof of Theorem~\ref{thm:energy}.

The finite-energy hypothesis is normally stronger than $\dimH\mu>k$. In Theorem~\ref{thm:energy}, strong separation is used only to formulate the dimension of the conditional measures. Under \SSC, one can use the entropy formula $\dimH\mu =\frac{h(p)}{\chi(p)}$ to deduce the fibre dimension $\dimH\mu-k$. Without \SSC, self-similar measures are still exact dimensional expressed in terms of projection entropy \cite{FengHu09}, but the dimension of the fibre measures are not clear due to potential overlaps in the fibres.

Both theorems hold for every $V\in\Gr(d,k)$, although the exceptional
sets of base points may depend on $V$. The approximation and tail
bounds in Theorem~\ref{thm:energy} are uniform in $V$, with constants
that may depend on the IFS, the weights, the energy exponent and
energy bound, $d$, and $k$. The tail bound implies a uniform positive
lower bound for projection volumes, while equivalence with restricted
Lebesgue measure allows statements holding $\mu_V$-almost everywhere
to be moved to Lebesgue-almost every point of $\pi_VA$.

The tail estimate \eqref{eq:main-tail} also provides uniform logarithmic
integrability of the projected densities:

\begin{corollary}\label{cor:bound}
Under the hypotheses of Theorem~\ref{thm:energy}, the projected densities
$f_V$ are uniformly integrable. Moveover, for every $q>0$,
\[
 \sup_{V\in\Gr(d,k)}
 \int_V f_V(\log^+f_V)^q\dd\LL_V^k<\infty,
\]
where $\log^+u=\max\{\log u,0\}$ for $u>0$, with
$\log^+0=0$.
\end{corollary}

The case $q=1$ of Corollary~\ref{cor:bound} also implies the following uniform entropy bound. Fix $D_A\ge1$ such that $A\subset B(0,D_A)$.
After identifying each $V$ isometrically with $\R^k$, let
$J=[-(D_A+1),D_A+1)^k$, and let $\mathcal D_n(J)$ be its partition into
$2^{kn}$ dyadic cubes of equal volume.
Let $\lambda_J=\LL^k|_J/|J|$.
For a probability measure $\nu$ supported on $J$, define
\[
 H(\nu,\mathcal D_n(J))
 =-\sum_{D\in\mathcal D_n(J)}\nu(D)\log\nu(D).
\]
We use the convention $0\log0=0$. We then have:

\begin{corollary}\label{cor:entropy}
Under the hypotheses of Theorem~\ref{thm:energy}, there exists a constant $C$ such that
\begin{equation}\label{eq:entropy}
 0\le kn\log2-H(\mu_V,\mathcal D_n(J))\le C
\end{equation}
for every $V\in\Gr(d,k)$ and integer $n\ge0$.
\end{corollary}

We note that the restriction $d\ge3$ is necessary for the theorems discussed above in
this generality. When $d = 2$, Rapaport \cite[Theorem~1.1]{Rapaport17} constructed
a planar homogeneous IFS with strong separation and an
irrational rotation whose uniform self-similar measure has dimension
greater than one. On a dense $G_\delta$ set of directions, its
projections are singular, its typical conditional measures are
discrete, and dimension conservation fails
\cite[Corollary~1.2]{Rapaport17}. Thus the planar analogue of
Theorem~\ref{thm:energy} is false. One may need more than dense rotation condition for planar self-similar measures. Questions about extensions of
the results are discussed in Section~\ref{sec:prospects}.

\subsection{Previous work and proof strategy}\label{sec:comparison}

We compare our results with earlier work on dimension conservation
and absolute continuity, and then explain the main steps of the proof.
Related results and questions for self-affine and stationary measures
are discussed in Section~\ref{sec:prospects}.

\medskip
\noindent\textit{Dimension conservation and fibres.}
Furstenberg \cite[Theorem~6.2]{Furstenberg08} established dimension
conservation for homogeneous fractals, including strongly separated
self-similar sets with no rotations. Falconer and the first author
\cite[Corollary~1.1]{FalconerJin14} proved the corresponding measure
statement in the finite-rotation setting without a separation
condition. For dense rotations, Rapaport
\cite[Theorem~3.1 and Corollary~3.2]{Rapaport20} proved that planar
homogeneous self-similar measures with strong separation and dimension
greater than one have projection densities in some $L^q$, $q>1$,
in every direction, provided the common rotation is irrational and
the contraction-rotation parameter avoids a set of Hausdorff dimension
zero. He also obtained exact dimensionality of the conditional
measures and equivalence of the projections with restricted Lebesgue
measure. The theorem excludes a zero-dimensional parameter set,
but does not provide an effective criterion for verifying that a
prescribed parameter lies outside it.
In Sections~\ref{sec:conditional} and~\ref{sec:equivalence},
we adapt the arguments of \cite[Sections~5 and~7]{Rapaport20}
to prove exact dimensionality of the conditional measures and
equivalence with Lebesgue measure restricted to the projected
attractor in every projection.

For sets, Falconer and the first author
\cite[Theorem~4.8]{FalconerJin15} obtained lower bounds
$s-1-\varepsilon$ for fibre dimensions on a positive-length set
of image points, outside a zero-dimensional exceptional set of
directions. Their result applies to planar self-similar sets of
dimension $s>1$ with dense rotations and the open set condition,
for every $0<\varepsilon<s-1$.
For planar homogeneous self-similar sets with an irrational rotation, Wu \cite[Theorem~1.6]{Wu19} and Shmerkin
\cite[Corollary~8.3]{Shmerkin19} proved for the upper box dimension
\[
 \overline{\dim}_{\mathrm B}(A\cap\ell)
 \le \max\{\dimH A-1,0\}
 \quad\text{for every affine line }\ell,
\]
assuming strong separation and the open set condition. Shmerkin asked for complementary lower bounds for many parallel
fibres in Remark~(b) following that corollary.
Theorem~\ref{thm:sets} implies the exact dimension $s-k$ at
Lebesgue-almost every point of every projected image in dimensions
$d\ge3$ under \DRC, allowing unequal ratios and overlaps, answering this question of Shmerkin when $d \geq 3$. However, it does not settle the $d = 2$ case.

Under strong separation, dense rotation and $\dimH\mu>k$,
Hochman and Shmerkin's projection theorem
\cite[Theorem~1.6]{HochmanShmerkin12} implies an $o(n)$ difference from maximal entropy
for each fixed projection. Under the hypotheses of
Theorem~\ref{thm:energy}, Corollary~\ref{cor:entropy} implies an $O(1)$
difference, uniformly over projection planes and scales.

\medskip
\noindent\textit{Absolute continuity and random walks.}
Lindenstrauss and Varj\'u \cite{LindenstraussVarju16} used
spectral-gap estimates for random Euclidean isometries to prove
smoothness of certain self-similar measures. The corresponding
rotation-averaging operators also appear in proving Fourier decay for
self-similar measures, see \cite[Section~4.1]{Sahlsten25}.
There, self-similarity relates a Fourier transform of $\mu$ at a large
frequency to its values at rotated directions and a smaller
frequency. Rotation estimates then compare this directional
average with the average over the rotation group allowing to use $L^2$ contraction of the associated transfer operator.

We are inspired by this strategy but crucially replace the Fourier transform of $\mu$ by
the $L^1$ difference between successive Gaussian smoothings
of a projected measure $\mu_V$. These smoothing differences play the
role of \textit{martingale differences}. The important point is that the self-similarity in our argument causes
no growing terms in the $L^1$ norm. This allows Varj\'u's $L^2$
estimate \cite[Corollary~7]{Varju13} for random walk operator to be used udner \DRC\ alone. This is the main novelty of our paper.

The closest comparison for projections is the recent work of Algom,
Rodriguez-Hertz and Wang \cite{AlgomRodriguezHertzWang26}.
Their Theorems~1.2 and~1.9 give criteria for absolute continuity
and Sobolev or Besov regularity in terms of quantitative mixing
and dimension. Their paper allows rotation orbits smaller than the full Grassmannian and uses Littlewood-Paley pieces in $L^q$, $q>1$.
In the equicontractive $\R^3$ setting, their absolute-continuity
criterion compares a spectral gap with the dimension of the self-similar measure and its contraction ratio. Our theorem assumes the full
$\SO(d)$ rotation group, but requires only
dense generation and finite energy above the target dimension.
We obtain quantitative $L^1$ control rather than the positive
Sobolev or Besov regularity that follow under their stronger
quantitative hypotheses.

\medskip
\noindent\textit{Proof strategy.}
Suppose that $\mu$ satisfies the hypotheses of
Theorem~\ref{thm:energy}. For clarity let us first assume equicontractive case $r_i = \rho$ for all $i = 1,\dots,m$. For $V\in\Gr(d,k)$, recall
$\mu_V=(\pi_V)_*\mu$, and that $\Gamma_h^V$ is
the normalized Gaussian on $V$ at scale $h$. Consider the averaging
operator $P$ acting on functions $F:\Gr(d,k)\to\C$ by
\[
 (PF)(V)=\sum_i p_iF(O_i^{-1}V), \quad V \in \Gr(d,k).
\]
For each $R\ge1$, define the following function
$U_R:\Gr(d,k)\to[0,\infty)$ by
\[
U_R(V):=\left\|\mu_V*(\Gamma_{1/R}^V-\Gamma_{2/R}^V)\right\|_{L^1(V)}=\int_V\left|\int_V\bigl(\Gamma_{1/R}^V(x-y)-\Gamma_{2/R}^V(x-y)\bigr)\,d\mu_V(y)\right|\,d\mathcal L^k_V(x).
\]
Thus $U_R(V)$ measures the change in smoothing of $\mu_V$ from scale $2/R$ to scale $1/R$.

For $k<t<\min\{t_0,k+2\}$, finite energy and Fourier averaging give
\[
 \int_{\Gr(d,k)}U_R(V)\dd\sigma_k(V)\le CR^{-(t-k)/2}.
\]
Gaussian translation estimates also give
$\Lip(U_R)\le CR$ and $\|U_R\|_\infty\le2$.
We would then like to use this bound for the average to be valid for \textit{every} $V \in \Gr(d,k)$.

Fixing $V \in \Gr(d,k)$, self-similarity implies the crucial bound
\begin{align}\label{eq:crucialbound}
    U_B(V)\le P^nU_{B\rho^n}(V)
\end{align}
whenever $B\rho^n\ge1$. This is analogous to the bound for Fourier transform $|\widehat{\mu}(\xi)|$ in
\cite{LindenstraussVarju16,Sahlsten25}. Here instead we apply self-similarity
directly to $U_R(V)$.
Under a branch with contraction $\rho^n$, the density factor
$\rho^{-kn}$ cancels the Jacobian $\rho^{kn}$ in the
$L^1$ norm. Thus there is therefore no growing factor in \eqref{eq:crucialbound}. By contrast, an $L^q$ norm with $q>1$ would
introduce the factor $\rho^{-nk(1-1/q)}$ in \eqref{eq:crucialbound}.

In Section~\ref{sec:mix}, we show that \DRC\ implies the
support condition of Varj\'u's Corollary~7 in \cite{Varju13} after passing to
a fixed convolution power. Iterating the corollary implies
$L^2$ decay and the Lipschitz bound then implies pointwise mixing on the rotation planes: there exist constants $C,c>0$ independent of $F$ and $n$ such that
\[
 \|P^nF-\sigma_k(F)\|_\infty
 \le Ce^{-cn^{1/3}}
       \bigl(\|F\|_\infty+\Lip(F)\bigr),
 \qquad
 \sigma_k(F)=\int_{\Gr(d,k)}F\dd\sigma_k.
\]
Combining this with \eqref{eq:crucialbound} for $F = U_{B\rho^n}$ and the estimates for $U_R$,
where $R=B\rho^n\ge1$, implies
\[
 \sup_VU_B(V)
 \le CR^{-(t-k)/2}+CRe^{-cn^{1/3}}.
\]
We may choose $n$ so that $R$ is of order
$\exp(a(\log B)^{1/3})$, with $a>0$ sufficiently small, to balance the two terms on the right. Then both terms' decay would be bounded by the same speed, hence
\[
 \sup_VU_B(V)\le C'e^{-c'(\log B)^{1/3}}
\]
for some constants $C',c'>0$. These bounds are summable over dyadic scales, so using standard analytic arguments the Gaussian
smoothings converge in $L^1$ to densities $f_V$, uniformly
over $V$. The same estimates give the approximation rate
\eqref{eq:main-rate} and the density-tail bound
\eqref{eq:main-tail}.

For general contraction ratios $r_i$, we consider words of a fixed length and separate them according to their symbol counts. Words with the same symbol counts have exactly the same contraction ratio. By Lemmas~\ref{lem:balanced} and~\ref{lem:block}, we can choose a count class (good block) for which the conditional rotations satisfy the support assumption in Varj\'u's corollary. We apply the mixing estimate each time such a good block appears. The rotation averages coming from the blocks in between these good blocks do not affect the bound since the operator is contracting. We continue until, by a stopping time, the remaining scale is of order $\exp(a(\log B)^{1/3})$, which implies the same smoothing estimate as in the equicontractive case described above. Since we only condition on the symbol counts and then sum over the count classes, the original weights of the measure are unchanged.

We finish the proof using the uniform bounds on the projected densities, following arguments of Rapaport~\cite{Rapaport20}. Uniform integrability and self-similarity show that each projected measure is equivalent to Lebesgue measure on the projected attractor. Under strong separation, the density-tail estimate and the formula for conditional measures on cylinders give exact dimensionality of the conditional measures for arbitrary positive weights.

For the set theorem, we use Farkas's approximation by strongly separated subsystems whose rotations form a dense subgroup, see \cite[Proposition~1.8]{Farkas16}. We apply the theorem to the natural self-similar measures on these subsystems and transfer the resulting fibre bounds to the original set.

\begin{remark}
After completing this work, we learned of independent work by Algom,
Rodriguez-Hertz and Wang \cite{AlgomRodriguezHertzWang26b}
on absolute continuity of projections of nonlinear self-conformal
measures in $\R^2$. Their proof uses $L^1$ norms of Littlewood-Paley
pieces but instead of Varj\'u's $L^2$ estimates, which would only work in dimension $d \geq 3$, they apply effective renewal-theoretic arguments. Together, \cite{AlgomRodriguezHertzWang26b} and our work suggest further avenue for
applications to absolute continuity questions,
for instance for stationary measures or dynamically defined spectral
measures, see Section~\ref{sec:prospects} for further discussion.
\end{remark}

\subsection*{Structure of the paper}

Section~\ref{sec:prelim} introduces the notation, symbolic coding,
and disintegration of measures.
Section~\ref{sec:mix} proves the needed mixing estimates using \cite{Varju13}. Section~\ref{sec:smoothing} proves absolute
continuity and the smoothing and density-tail bounds for finite-energy
self-similar measures with equal contraction ratios.
Section~\ref{sec:unequal} extends these results to arbitrary contraction ratios.
Section~\ref{sec:entropy} proves the Corollary~\ref{cor:bound} and Corollary~\ref{cor:entropy}. 
Section~\ref{sec:conditional} proves exact dimensionality of the
conditional measures and dimension conservation under \SSC.
Section~\ref{sec:equivalence} proves equivalence with restricted
Lebesgue measure and the uniform positive projection-volume bound,
completing the proof of Theorem~\ref{thm:energy}.
Section~\ref{sec:generalsets} uses natural self-similar measures on strongly
separated subsystems to prove Theorem~\ref{thm:sets}. Section~\ref{sec:prospects} discusses extensions and open questions.

\subsection*{AI declaration}
AI tools were only used to assist with checking correctness, notations and cross-referencing. The ideas and mathematical arguments are those of the authors.
The authors wrote and checked the manuscript and take full responsibility for its content.

\section{Preliminaries}\label{sec:prelim}
\subsection{Notation and symbolic coding}\label{subsec:notation}
We use the notation of the IFS \eqref{eq:ifs}, its attractor
$A$, and the self-similar measure $\mu$ with positive
weights $p_i$. The assumptions \SSC\ and \DRC\ are assumed only
where stated. We denote by $\nu\ll\eta$ for absolute continuity and
$\nu\sim\eta$ when mutula absolute-continuity hold.
For a Borel probability measure $\nu$, we use the definition
\[
 \dimH\nu=\inf\{\dimH E:E\text{ is Borel and }\nu(E)=1\}
\]
for the (upper) Hausdorff dimension of $\mu$. For an exact-dimensional probability measure, this equals its almost-sure local dimension and its lower Hausdorff dimension. See \cite{Mattila95}. Constants denoted by $C,c$ may
vary between occurrences. They never depend on the projection
plane or the running scale.

Let $\N=\{1,2,\ldots\}$ and $\N_0=\N\cup\{0\}$. Let $\mathcal I=\{1,\ldots,m\}$,
$\Omega=\mathcal I^{\N}$, and $\Pp=p^{\N}$, where the probability vector
$p=(p_1,\ldots,p_m)$ has positive coordinates.
For $\omega\in\Omega$ define the coding map
\[
 \Pi(\omega)=\lim_{n\to\infty}
    S_{\omega_1}\circ\cdots\circ S_{\omega_n}(0).
\]
Its image under $\Omega$ is the self-similar set $A$, and the self-similar measure takes the form $\mu=\Pi_*\Pp$. Since $p_i$ are all positve, every open neighbourhood of any point of $A$ contains a sufficiently small cylinder of positive mass, hence $\supp\mu=A$.

For $w=i_1\cdots i_n\in\mathcal I^n$ we will use the following notations:
\[
 \begin{aligned}
  S_w&=S_{i_1}\circ\cdots\circ S_{i_n}
      =r_wO_w(\cdot)+a_w,\qquad O_w=O_{i_1}\cdots O_{i_n},\\
  r_w&=\prod_{j=1}^n r_{i_j},\qquad
  p_w=\prod_{j=1}^n p_{i_j},\qquad A_w=S_wA.
 \end{aligned}
\]
Let $\sigma:\Omega\to\Omega$ be the left-shift operator,
$\sigma(\omega_1,\omega_2,\ldots)=(\omega_2,\omega_3,\ldots)$.
The coding map satisfies
$\Pi(\omega)=S_{\omega|n}(\Pi(\sigma^n\omega))$.
Under \SSC, every point of $A$ has a unique coding,
and
\begin{equation}\label{eq:restriction}
 \mu|_{A_w}=p_w(S_w)_*\mu,\qquad \mu(A_w)=p_w.
\end{equation}
Without \SSC\ one still has the inequality $\mu\ge p_w(S_w)_*\mu$, which follows by iterating the self-similarity identity $\mu=\sum_i p_i(S_i)_*\mu$.

\subsection{Disintegration and conditional probability measures}\label{subsec:disintegration}
For a fixed $V\in\Gr(d,k)$, by the disintegration theorem (see \cite{Kallenberg21} for example) one can find a Borel set $B_0\subset V$ with $\mu_V(B_0)=1$ and a Borel probability kernel
$y\mapsto\mu_{V,y}$ defined on $B_0$, such that
\begin{equation}\label{eq:disintegration}
 \mu(E)=\int_{B_0}\mu_{V,y}(E)\dd\mu_V(y),\qquad
 \mu_{V,y}(A\cap\pi_V^{-1}\{y\})=1\quad(y\in B_0)
\end{equation}
for every Borel $E\subset\R^d$. We leave the conditional measures
undefined on $V\setminus B_0$. Integrals involving the conditional
measures are understood to be taken over $B_0$.
Existence and uniqueness up to a $\mu_V$-null set also follow from the disintegration theorem. Whenever a full-$\mu$-measure set is disintegrated, it has full $\mu_{V,y}$-measure for $\mu_V$-almost every $y$, by
\eqref{eq:disintegration}.
Whenever the projected measures have densities, the Borel carrying set
can be chosen simultaneously for a countable family of planes.

\section{Mixing on the group and the Grassmannian}\label{sec:mix}

We will establish the mixing estimate used in
Section~\ref{sec:smoothing}. We first work with functions on the
rotation group, then pass to functions of projection planes.

\subsection{Averaging operator and contraction estimates}
Assume \DRC, and let $G=\SO(d)$, with the Haar probability
measure $\vartheta$. Note that $G$ is compact, connected,
and semisimple for $d\ge3$. We use the bi-invariant Frobenius metric
\[
 \delta(g,h)=\|g-h\|_{\Frob},
 \qquad
 \|M\|_{\Frob}^2=\sum_{i,j=1}^d M_{ij}^2.
\]
All $L^q$ norms on $G$ are taken with respect to $\vartheta$.
For a Lipschitz function $F:G\to\C$, denote by
\[
 \Lip(F)=\sup_{g\ne h}
 \frac{|F(g)-F(h)|}{\delta(g,h)}
\]
its Lipschitz constant. For a probability vector $p$ with postive entries define the rotation law $\kappa$ associated with the rotaions $\{O_1,\cdots,O_m\}$ and its averaging operator $Q$ by
\[
 \kappa=\sum_i p_i\delta_{O_i},
 \qquad
 (QF)(g)=\int_G F(h^{-1}g)\dd\kappa(h)
       =\sum_i p_iF(O_i^{-1}g).
\]
Denote by $\vartheta(F)=\int_G F\dd\vartheta$.
The operator $Q$ preserves constants and satisfies
\begin{equation}\label{eq:contr}
     \vartheta(QF)=\vartheta(F),\qquad
 \|QF\|_q\le\|F\|_q\quad(q=2,\infty),\qquad
 \Lip(QF)\le\Lip(F).
\end{equation}
These properties follow from the Haar invariance, Jensen's inequality,
and the fact that left translations are isometries.
In particular, iteration on $Q$ preserves the mean-zero condition and
does not increase $\|F\|_\infty+\Lip(F)$, which we will use in the proof of Proposition~\ref{prop:mix}.

For probability measures $\eta_1,\eta_2$ on $G$, the convolution of $\eta_1$ and $\eta_2$ is defined as
\[
 \int_G F(u)\dd(\eta_1*\eta_2)(u)
 =\iint_{G\times G}F(gh)\dd\eta_1(g)\dd\eta_2(h).
\]
Thus $\kappa^{*b}$, that is $\kappa$ convoluted to itself $b-1$ times, is the law of a product of $b$ independent random rotations with law $\kappa$.

\subsection{Varj\'u's \texorpdfstring{$L^2$}{L2} weak gap theorem}

We first present Varj\'u's $L^2$ weak gap theorem under the dense subgroup condition. Then we show that it applies to $\kappa^{*b}$ of a fixed length $b$.

For a probability measure $\eta$ on $G$, let
$\widetilde\eta$ be its image under inversion, and recall its average operator
\[
 Q_\eta F(g)=\int_G F(h^{-1}g)\dd\eta(h).
\]
In particular, $Q_{\kappa^{*b}}=Q^b$.
The following is \cite[Corollary~7]{Varju13}, with the
logarithmic exponent replaced by its upper bound $2$.
The bi-invariant Riemannian metric used there is comparable
with our Frobenius metric, so this changes only the constant.

\begin{theorem}\label{thm:varju}
Let $\eta$ be a probability measure on $G$ such that
$\supp(\widetilde\eta*\eta)$ generates a dense subgroup of $G$.
There is a constant $c_\eta>0$ such that every Lipschitz
$\varphi:G\to\C$ with $\vartheta(\varphi)=0$ and
$\|\varphi\|_2=1$ satisfies
\begin{equation}\label{eq:varju}
 \|Q_\eta\varphi\|_2
 \le 1-\frac{c_\eta}
 {\log^2\!\bigl(2+\|\varphi\|_\infty+\Lip(\varphi)\bigr)}.
\end{equation}
\end{theorem}

The dense subgroup condition is not necessary to hold for $\kappa$ itself:
dense generation by the rotations does not imply dense generation
by their pairwise differences. The next lemma shows that we can instead use the finite product law of independent rotations with law $\kappa$.

\begin{lemma}\label{lem:block}
Assume \DRC. Then there is an integer $b\ge1$, depending only
on the rotations $O_1,\ldots,O_m$, such that
\[
 \overline{\left\langle
 \supp\bigl(\widetilde{\kappa^{*b}}*\kappa^{*b}\bigr)
 \right\rangle}=G.
\]
\end{lemma}

\begin{proof}
We will use the following result of Breuillard-Gelander
\cite[Corollary~2.5]{BreuillardGelander03}: every dense subgroup of a
connected real Lie group $\mathbf{g}$ satisfying \textit{perfectness}, i.e. the commutator group $\overline{[\mathbf{g},\mathbf{g}]}=\mathbf{g}$,
contains a finitely generated dense subgroup. We will construct
an increasing sequence of subgroups with dense union, then apply this result to
obtain a finite collection of dense generators. Define
\[
 \Gamma_n=\langle O_u^{-1}O_v:u,v\in\mathcal I^n\rangle,\qquad
 \Gamma_\infty=\bigcup_{n\geq1}\Gamma_n,\qquad
 H=\overline{\Gamma_\infty}.
\]
Note that, by the positivity of the weights in $\kappa$, the support
\[
 \operatorname{supp}\bigl(\widetilde{\kappa^{*n}}*\kappa^{*n}\bigr)
 =\Gamma_n.
\]
For $i\in\mathcal I$ and $u,v\in\mathcal I^n$ we have
\[
 O_{iu}^{-1}O_{iv}=O_u^{-1}O_v,\qquad
 O_{ui}^{-1}O_{vi}=O_i^{-1}(O_u^{-1}O_v)O_i.
\]
This implies that $\Gamma_n\subset\Gamma_{n+1}$ and
$O_i^{-1}\Gamma_n O_i\subset\Gamma_{n+1}$. Thus $\Gamma_\infty$ is a
subgroup of $G$ and $O_i^{-1}HO_i\subset H$ holds for every index $i\in\mathcal I$.

Fix $i\in\mathcal I$. By compactness one can find a subsequence $n_j$ such that $O_i^{n_j}\to a$ for some $a\in G$. We can then choose the sequence so that $n_{j+1}-n_j\to\infty$, and for $q_j:=n_{j+1}-n_j$ we have $O_i^{q_j}=O^{n_{j+1}}O^{-n_j}\to e$ . Iterate $r$ times $O_i^{-1}HO_i\subset H$ we get $O_i^{-r}HO_i^r\subset H$ for every $r\geq1$.
Since $O_i^{-(q_j-1)}\to O_i$ and $H$ is closed, we have
$O_iHO_i^{-1}\subset H$. Together with $O_i^{-1}HO_i\subset H$ we obtain $O_iHO_i^{-1}=H$. The same equality holds for every element of
$\langle O_1,\ldots,O_m\rangle$, and hence, by \DRC\ and continuity,
for every element of $G$. Therefore $H$ is normal in $G$. For
$i,j\in\mathcal I$, we have $O_i^{-1}O_j\in\Gamma_1\subset H$, so all
the rotations $O_i$ have the same image in $G/H$. This quotient therefore
has a dense cyclic subgroup and is abelian, thus $[G,G]\subset H$. Since $G=\SO(d)$ is perfect for $d\ge 3$, one gets $H=G$.

Now apply \cite[Corollary~2.5]{BreuillardGelander03} to the dense subgroup $\Gamma_\infty$. We obtain finitely many elements $h_1,\ldots,h_\ell\in\Gamma_\infty$ whose generated subgroup is dense in $G$. Since $\Gamma_\infty=\bigcup_{n\geq1}\Gamma_n$ and the groups $\Gamma_n$ are increasing, there is a finite integer $b$ such that $h_1,\ldots,h_\ell\in\Gamma_b$. This means $\langle h_1,\ldots,h_\ell\rangle\subset \Gamma_b.$ On the other hand the subgroup on the left is dense in $G$, so $\overline{\Gamma_b}=G$, which yields the conclusion. Note that the integer $b$ only depends on $O_1,\ldots,O_m$, not on the probability vector $p$.
\end{proof}

\subsection{Pointwise mixing}

We now iterate Varj\'u's weak gap theorem to firstly obtain decay in $L^2$, then combined with the Lipschitz bound we will obtain the \textit{uniform} pointwise estimate:

\begin{proposition}\label{prop:mix}
There are constants $C,c>0$, depending only on $G$ and
$\kappa$, such that for every Lipschitz $F:G\to\C$ and
integer $n\ge0$,
\begin{equation}\label{eq:groupmix}
 \|Q^nF-\vartheta(F)\|_\infty
 \le Ce^{-cn^{1/3}}
       \bigl(\|F\|_\infty+\Lip(F)\bigr),
\end{equation}
where $\vartheta(F)=\int_G F\dd\vartheta$.
\end{proposition}

\begin{proof}
By subtracting the mean $\vartheta(F)$ and rescaling, it is enough to prove
\eqref{eq:groupmix} when
\[
 \vartheta(F)=0,\qquad
 \|F\|_\infty+\Lip(F)\le1.
\]
Take the integer $b$ as in Lemma~\ref{lem:block}, and shortly denote by $T=Q^b$.
Let $a_j:=\|T^jF\|_2$ for $j\ge 0$. By \eqref{eq:contr} we have
\[
 \|T^jF\|_\infty+\Lip(T^jF)\le1,\qquad
 a_{j+1}\le a_j\le1.
\]

Fix $0<\varepsilon<1/2$. Whenever $a_j>\varepsilon$, the function $T^jF/a_j$
has mean zero, $L^2$ norm one, and
\[
 \|T^jF/a_j\|_\infty+\Lip(T^jF/a_j)=a_j^{-1}(\|T^jF\|_\infty+\Lip(T^jF))
 \le\varepsilon^{-1}.
\]
By Lemma~\ref{lem:block}, we can apply \eqref{eq:varju} with $\eta=\kappa^{*b}$ to obtain
\[
a_{j+1}/a_j
=\left\|T\left(T^jF/a_j)\right)\right\|_2
\le
1-\frac{c_\eta}
{\log^2\!\bigl(2+\|T^jF/a_j\|_\infty+
\operatorname{Lip}(T^jF/a_j)\bigr)}.
\]
where recall $Q_\eta=Q^b=T$. Now
\[
\log\!\bigl(2+\|T^jF/a_j\|_\infty+
\operatorname{Lip}(T^jF/a_j)\bigr)
\le \log(2+\varepsilon^{-1}):=L_\varepsilon.
\]
Hence
\begin{equation}\label{eq:mix-recurrence}
 a_{j+1}\le a_j\left(1-\frac{c_\eta}{L_\varepsilon^2}\right).
\end{equation}
Since $(a_j)_{j\ge0}$ is nonincreasing, if $a_j>\varepsilon$, then
\eqref{eq:mix-recurrence} applies to all $a_0\ge \cdots\ge a_{j-1}$. Thus
\[
 a_j\le \left(1-\frac{c_\eta}{L_\varepsilon^2}\right)^j \le \exp\left(-\frac{c_\eta j}{L_\varepsilon^2}\right).
\]
This means that for every $a_j$, either the above inequality holds or, since $0<\varepsilon<1/2$, we have $2\varepsilon^2+\epsilon< 1$, thus
\[
a_j\le \varepsilon \le (2+\varepsilon)^{-1/2}=e^{-L_\varepsilon/2}.
\]
We therefore deduce for all $j\ge 0$,
\[
a_j\le \max\left\{e^{-L_\varepsilon/2},e^{-c_\eta j/L_\varepsilon^2}\right\}.
\]
To balance the two terms on the right-hand side we may take $L_\varepsilon=(2c_\eta j)^{1/3}$ for $j\ge (\log 4)^3/(2c_\eta)$. Thus, write $c_0=(2c_\eta)^{1/3}/2>0$, we have $a_j\le \exp(-c_0j^{1/3})$ for all $j\ge (\log 4)^3/(2c_\eta)$. Then by absorbing the terms for $j<(\log 4)^3/(2c_\eta)$ to the prefactor one can find constant $C_0$ such that
\begin{equation}\label{eq:mix-ltwo}
 \|T^jF\|_2=a_j\le C_0e^{-c_0j^{1/3}}
\end{equation}
holds for all $j\ge 0$.

For the supremum norm, shortly denote by $H=T^jF$ and
$s=\|H\|_\infty$. If $s>0$, choose $g_0\in G$ with
$|H(g_0)|=s$. Since $\Lip(H)\le1$, we have
$|H(g)|\ge s/2$ on the ball $B_G(g_0,s/2)$.
Haar measure is a smooth volume on the compact manifold $G$,
so, writing $d_G=\dim G$, there exists a constant $c_G>0$ such that
\[
 \vartheta(B_G(g,r))\ge c_G r^{d_G}
 \qquad(g\in G,\ 0<r\le1).
\]
Thus we can bound:
\[
\begin{aligned}
a_j^2
=\|H\|_2^2
&=\int_G|H(g)|^2\,d\vartheta(g)\\
&\ge \int_{B_G(g_0,s/2)}|H(g)|^2\,d\vartheta(g)\\
&\ge \left(\frac{s}{2}\right)^2
      c_G\left(\frac{s}{2}\right)^{d_G}\\
&=\,c_1s^{d_G+2},
\end{aligned}
\]
where $c_1=c_G/2^{d_G+2}$. This implies
\[
\|T^jF\|_\infty=s
\le c_1^{-1/(d_G+2)}a_j^{2/(d_G+2)}.
\]
By \eqref{eq:mix-ltwo} we obtain
\[
\begin{aligned}
\|T^jF\|_\infty
&\le c_1^{-1/(d_G+2)}
\left(C_0e^{-c_0j^{1/3}}\right)^{2/(d_G+2)}\\
&=c_1^{-1/(d_G+2)}C_0^{2/(d_G+2)}
\exp\left(-\frac{2c_0}{d_G+2}j^{1/3}\right)\\
&=C_2e^{-c_2j^{1/3}},
\end{aligned}
\]
where
\[
C_2=c_1^{-1/(d_G+2)}C_0^{2/(d_G+2)}<\infty \text{ and } c_2=\frac{2c_0}{d_G+2}>0.
\]
The same bound holds immediately if $s=0$.

Finally, write $n=bj+r$, with $0\le r<b$. Use \eqref{eq:contr} again we have
\[
 \|Q^nF\|_\infty
 =\|Q^rT^jF\|_\infty
 \le C_2e^{-c_2j^{1/3}}.
\]
For $n\ge2b$, we have $j\ge n/(2b)$ hence $e^{-c_2j^{1/3}}\le e^{-c_2(2b)^{-1/3}n^{1/3}}$. Let $c_3=c_2/(2b)^{-1/3}>0$ and take $C_3$ large enough to absorb the first $2b$ terms, we get for all $n\ge 0$,
\[
 \|Q^nF\|_\infty \le C_3e^{-c_3n^{1/3}}.
\]
\end{proof}

\begin{remark}\label{rem:mix-exponent}
The exponent $1/3$ comes from balancing the two terms
$e^{-L}$ and $e^{-j/L^2}$ for large $j$, hence $L\approx j^{1/3}$. More generally, if the squared
logarithm in \eqref{eq:varju} is replaced by a power $\beta>0$ then the exponent becomes $1/(\beta+1)$.
\end{remark}

\subsection{Pointwise mixing of projection planes}

The smoothing estimates we will use soon to prove Theorem \ref{thm:energy} concern functions of a projection plane, rather than functions of a rotation.
We will obtain the corresponding mixing estimate by fixing one plane
and following its images under rotations.

Equip the Grassmannian $\Gr(d,k)$ with its rotation-invariant probability
measure $\sigma_k$ and the quotient Frobenius metric
\[
 d_{\Gr}(V,W)
 =\inf\{\|O-I\|_{\Frob}:O\in\SO(d),\ OV=W\}.
\]
This metric is bi-Lipschitz equivalent to the standard projection
metric $\|\pi_V-\pi_W\|_{\op}$. Define the average operator $P$ similar to $Q$ by
\[
 (PF)(V)=\sum_i p_iF(O_i^{-1}V).
\]
Thus $PF(V)$ averages $F$ over the planes obtained from $V$
by one random inverse rotation.

Fix $V_0\in\Gr(d,k)$. Given a Lipschitz function $F$ on
$\Gr(d,k)$, define
\[
 f(g)=F(gV_0),\qquad g\in G.
\]
Every $k$-plane is a rotation of $V_0$, so
$\|f\|_\infty=\|F\|_\infty$. Moreover, the rotation
$hg^{-1}$ carries $gV_0$ to $hV_0$, and hence
\[
 d_{\Gr}(gV_0,hV_0)
 \le\|hg^{-1}-I\|_{\Frob}
 =\|h-g\|_{\Frob}.
\]
It follows that $\Lip(f)\le\Lip(F)$.
The law of $gV_0$ under $\vartheta(\dd g)$ is exactly $\sigma_k$. Hence
\[
 \vartheta(f)
 =\int_{\Gr(d,k)}F\dd\sigma_k=:\sigma_k(F).
\]

The two averaging operators $P$ and $Q$ agree under this identification:
\[
 (Qf)(g)
 =\sum_i p_iF(O_i^{-1}gV_0)
 =(PF)(gV_0).
\]
In particular $Q^nf(g)=(P^nF)(gV_0)$ for $n\ge 1$. Since $gV_0$ ranges over every projection plane,
\[
 \|P^nF-\sigma_k(F)\|_\infty
 =\|Q^nf-\vartheta(f)\|_\infty.
\]
Proposition~\ref{prop:mix} therefore implies the following pointwise mixing estimate for projection planes:

\begin{corollary}
 Assume \DRC\ and let $1\le k<d$. There are constants $C,c>0$,
 depending only on the rotation law $\kappa$, $d$, and $k$, such that for every
 Lipschitz function $F:\Gr(d,k)\to\C$ and every integer $n\ge0$,
 \begin{equation}\label{eq:grassmix}
 \|P^nF-\sigma_k(F)\|_\infty
 \le Ce^{-cn^{1/3}}
       \bigl(\|F\|_\infty+\Lip(F)\bigr).
\end{equation}
\end{corollary}

\begin{remark}
The planar case is different because $\SO(2)$ is abelian.
Our argument uses that $\SO(d)$ has no nontrivial abelian quotient,
while Varj\'u's Corollary~7 requires a semisimple connected component.
Neither condition holds for $\SO(2)$.
In $\mathbb{R}^2$, a single irrational rotation $O$ generates a dense
subgroup of $\SO(2)$, but satisfies
\[
 (PF)(V)=F(O^{-1}V).
\]
Its iterations simply rotate the function, so
$\|P^nF-\sigma_1(F)\|_\infty
=\|F-\sigma_1(F)\|_\infty$. Thus dense generation alone cannot provide the mixing
estimate \eqref{eq:grassmix} in $\mathbb{R}^2$.
\end{remark}

\section{Absolute continuity for equal contraction ratios}\label{sec:smoothing}
Throughout this section, we work with the IFS \eqref{eq:ifs} with
$r_i=\rho$ for every $1\le i\le m$, and assume \DRC. Let
$\mu=\sum_i p_i(S_i)_*\mu$ be the associated self-similar measure, where $p_i>0$ and $\sum_i p_i=1$. Denote by $\chi=\log(1/\rho)$ and assume
\begin{equation}\label{eq:energy}
 I_{t_0}(\mu)=\iint|x-z|^{-t_0}\dd\mu(x)\dd\mu(z)<\infty
 \quad\text{for some }k<t_0<d.
\end{equation}
Choose
\begin{equation}\label{eq:tchoice}
 k<t<\min(t_0,k+2).
\end{equation}
Then $I_t(\mu)\le I_{t_0}(\mu)+1<\infty$. Let $\alpha=(t-k)/2\in(0,1)$.
Recall that $D_A>0$ is such that $\supp\mu=A\subset B(0,D_A)$.
For $R\ge1$, define $K_R^V=\Gamma_{1/R}^V-\Gamma_{2/R}^V$
and recall that we defined
\[
 U_R(V)=\|\mu_V*K_R^V\|_{L^1(V)}.
\]

\subsection{Estimates for \texorpdfstring{$U_R$}{U R}}

We first prove following estimates for $U_R$.
\begin{lemma}\label{lem:observable}
Uniformly in $R\ge1$,
\begin{equation}\label{eq:three}
 0\le U_R\le2,\qquad
 \Lip(U_R)\le CR,\qquad
 \int_{\Gr(d,k)}U_R\dd\sigma_k\le CR^{-\alpha}.
\end{equation}
\end{lemma}
\begin{proof}
The first estimate follows because every convolution of a probability
measure with a Gaussian kernel has integral one. For the Lipschitz bound, choose $O\in\SO(d)$ with $OV=W$, write $h_V=\mu_V*K_R^V$ and $h_W=\mu_W*K_R^W$.
Since $O^{-1}\pi_W=\pi_VO^{-1}$ and
$K_R^W(Oz)=K_R^V(z)$, we have
\[
 h_W(Oy)=\int_{\R^d}K_R^V(y-\pi_VO^{-1}x)\dd\mu(x),
 \qquad y\in V.
\]
Since the isometry $O:V\to W$ preserves Lebesgue measure,
$U_R(W)=\|h_W\circ O\|_{L^1(V)}$, so both norms can be
compared on $V$. For $F\in C^1(V)$ with $F,\nabla F\in L^1(V)$ one has
\[
 \|F(\cdot-u)-F(\cdot-v)\|_1\le |u-v|\|\nabla F\|_1.
\]
Moreover, $\nabla\Gamma_h^V(y)=-h^{-2}y\Gamma_h^V(y)$, so we can write
$\|\nabla\Gamma_h^V\|_1=h^{-1}\E|Z|$, where $Z$ is a standard
$k$-dimensional Gaussian random vector. Thus $\|\nabla K_R^V\|_1\le C_kR$, and
\begin{align*}
 |U_R(W)-U_R(V)|
 &\le\|h_W\circ O-h_V\|_{L^1(V)}\\
 &\le\int_{\R^d}
       \|K_R^V(\cdot-\pi_VO^{-1}x)-K_R^V(\cdot-\pi_Vx)\|_1
       \dd\mu(x)\\
 &\le C_kR\int_{\R^d}|\pi_V(O^{-1}-I)x|\dd\mu(x)\\
 &\le C_kD_AR\|O-I\|_{\Frob}.
\end{align*}
Here we used $\supp\mu\subset B(0,D_A)$ and
$\|O^{-1}-I\|_{\op}\le\|O^{-1}-I\|_{\Frob}=\|O-I\|_{\Frob}$.
Taking the infimum over $O$ we have
$|U_R(W)-U_R(V)|\le C_kD_AR\,d_{\Gr}(V,W)$, proving the Lipschitz bound.

For the third estimate, recall the Fourier transform $\widehat\mu(\xi):=\int e^{-2\pi i x\cdot\xi}\dd\mu(x)$.
Denote by $|\mathbb S^{j-1}|=\HH^{j-1}(\mathbb S^{j-1})$ the
surface area of the unit sphere in $\R^j$. The identity
\begin{equation}\label{eq:integralgeometry}
 \int_{\Gr(d,k)}\int_V F(\xi)\dd\LL_V^k(\xi)\dd\sigma_k(V)
 =\frac{|\mathbb S^{k-1}|}{|\mathbb S^{d-1}|}
     \int_{\R^d}F(\xi)|\xi|^{k-d}\dd\xi
\end{equation}
holds for every nonnegative Borel function $F:\R^d\to[0,\infty]$. Therefore, we have the following energy representation using Fourier transform (see \cite[Lemma~12.12]{Mattila95} for example):
\begin{equation}\label{eq:riesz}
 I_t(\mu)=c_{d,t}\int_{\R^d}
                |\widehat\mu(\xi)|^2|\xi|^{t-d}\dd\xi,\qquad
 c_{d,t}=\pi^{t-d/2}
             \frac{\Gamma((d-t)/2)}{\Gamma(t/2)}.
\end{equation}
Let $b(u)=e^{-2\pi^2u^2}-e^{-8\pi^2u^2}$.
For $\xi\in V$, we have
$\widehat{\mu_V}(\xi)=\widehat\mu(\xi)$ and
$\widehat{K_R^V}(\xi)=b(|\xi|/R)$.
One has $h_V\in L^1(V)\cap L^2(V)$, and by
Plancherel's identity,
\[
 \|h_V\|_2^2
 =\int_V|\widehat\mu(\xi)|^2|b(|\xi|/R)|^2\dd\LL_V^k(\xi).
\]
By \eqref{eq:integralgeometry} we get
\[
 \int_{\Gr(d,k)}\|h_V\|_2^2\dd\sigma_k(V)
 =\frac{|\mathbb S^{k-1}|}{|\mathbb S^{d-1}|}
       \int_{\R^d}|\widehat\mu(\xi)|^2
       |\xi|^{k-d}|b(|\xi|/R)|^2\dd\xi.
\]
Since $b(u)=6\pi^2u^2+O(u^4)$ at zero and decays exponentially
at infinity, we have
\[
 C_b:=\sup_{u>0}u^{k-t}|b(u)|^2<\infty
 \qquad(t\le k+4).
\]
This range includes our choice of $t$ in \eqref{eq:tchoice}.
Thus, for $\xi\ne0$,
\[
 |\xi|^{k-d}|b(|\xi|/R)|^2
 =R^{k-t}|\xi|^{t-d}(|\xi|/R)^{k-t}|b(|\xi|/R)|^2
 \le C_bR^{k-t}|\xi|^{t-d}.
\]
Using \eqref{eq:riesz} we get
\begin{equation}\label{eq:meanL2}
  \int_{\Gr(d,k)}\|h_V\|_2^2\dd\sigma_k(V)
 \le C R^{k-t} I_t(\mu).
\end{equation}

For $y\in V$ and $r>0$, write $B_V(y,r)=\{z\in V:|z-y|<r\}$.
To pass to the $L^1$ norm of $\mu_V*K_R^V$, using Cauchy-Schwarz on $B_V(0,D_A+1)$ one obtains
\[
 \int_{B_V(0,D_A+1)}|h_V(y)|\dd y\le C_{k,D_A}\|h_V\|_2.
\]
Outside this ball, $|y-\pi_Vx|\ge1$ for $x\in\supp\mu$, together with $|K_R^V|\le\Gamma_{1/R}^V+\Gamma_{2/R}^V$ we get
\[
 \int_{V\setminus B_V(0,D_A+1)}|h_V(y)|\dd y
 \le\int_{\{|z|\ge1\}}\bigl(\Gamma_{1/R}^V(z)+\Gamma_{2/R}^V(z)\bigr)\dd z.
\]
For a standard Gaussian $Z\in\R^k$, the Gaussian integral stands
$\E e^{|Z|^2/4}=2^{k/2}$, so by Markov's inequality,
\[
 \int_{\{|z|\ge1\}}\Gamma_h^V(z)\dd z
 =\Pr\{|Z|\ge h^{-1}\}
 \le 2^{k/2}e^{-1/(4h^2)}.
\]
Applying this with $h=1/R$ and $h=2/R$ we get
\[
 U_R(V)\le C_{k,D_A}\|h_V\|_2+2^{k/2+1}e^{-R^2/16}.
\]
Finally, by Cauchy-Schwarz with respect to $\sigma_k$ and \eqref{eq:meanL2}, we have
\begin{align*}
 \int_{\Gr(d,k)}U_R(V)\dd\sigma_k(V)
 &\le C_{k,D_A}\left(\int_{\Gr(d,k)}\|h_V\|_2^2\dd\sigma_k(V)\right)^{1/2}
       +2^{k/2+1}e^{-R^2/16}\\
 &\le C I_t(\mu)^{1/2}R^{(k-t)/2}+2^{k/2+1}e^{-R^2/16}
 \le CR^{-\alpha}.
\end{align*}
The last inequality uses $\alpha=(t-k)/2$, finiteness of
$I_t(\mu)$, and $\sup_{R\ge1}R^\alpha e^{-R^2/16}<\infty$.
\end{proof}

The proof of Lemma~\ref{lem:observable} uses only compact support of $\mu$ and
$I_t(\mu)<\infty$, with $k<t<\min(d,k+2)$. This lemma therefore
applies to any compactly supported probability measure satisfying
the same energy condition, as needed in Section~\ref{sec:unequal}.

\subsection{Using self-similarity and pointwise mixing}\label{subsec:self-similarity-mixing}

By \eqref{eq:three}, the observable $U_R$ has a small average over
projection planes when $R$ is large. Our aim is to turn this
averaged estimate into a bound for $U_B(V)$ that holds uniformly
in $V$. We first use self-similarity to compare $U_B$ with a
rotation average at the remaining scale $R=B\rho^n$, and then
apply pointwise mixing. Using the word notation from
Section~\ref{subsec:notation}, denote by $V_w=O_w^{-1}V$ for
$w\in\mathcal I^n$. We will show that, whenever $B\rho^n\ge1$,
\begin{equation}\label{eq:recursion}
U_B(V)\le\sum_{|w|=n}p_w U_{B\rho^n}(V_w)
=P^n U_{B\rho^n}(V).
\end{equation}
The important point is that this comparison has no scaling loss
in $L^1$.

To prove \eqref{eq:recursion}, define the similarity map between $k$-planes
$S_{V,w}:V_w\to V$ by
\[
S_{V,w}(z)=\pi_Va_w+\rho^n O_wz.
\]
By definition one has $(\pi_V\circ S_w)_*\mu=(S_{V,w})_*\mu_{V_w}$.
Therefore, by iterating $n$ times the idenity $\mu=\sum_{i} p_i(S_i)_*\mu$ and using the triangle inequality, one gets
\[
U_B(V)\le
\sum_{|w|=n}p_w
\bigl\|((S_{V,w})_*\mu_{V_w})*K_B^V\bigr\|_1.
\]
It remains to identify each norm on the right with
$U_{B\rho^n}(V_w)$. Convolution with a Gaussian kernel of width $h$ pulls back to a Gaussian kernel of width $h/\rho^n$. More precisely, for $y\in V$ and $r=\rho^n$,
\[
\bigl(((S_{V,w})_*\mu_{V_w})*K_B^V\bigr)(y)
=r^{-k}\bigl(\mu_{V_w}*K_{Br}^{V_w}\bigr)
\bigl(r^{-1}O_w^T(y-\pi_Va_w)\bigr).
\]
To verify the identity, write $z=r^{-1}O_w^T(y-\pi_Va_w)$. Its left-hand side is
\[
\int_{V_w}K_B^V(y-\pi_Va_w-rO_wu)\dd\mu_{V_w}(u)
=\int_{V_w}K_B^V(rO_w(z-u))\dd\mu_{V_w}(u).
\]
Then we may prove the identity by
$\Gamma_{1/B}^V(rO_wv)=r^{-k}\Gamma_{1/(Br)}^{V_w}(v)$
and the analogous identity to $2/B$.
In the change of variables $y=\pi_Va_w+rO_wz$,
$\dd\LL_V^k(y)=r^k\dd\LL_{V_w}^k(z)$.
Thus the density factor $r^{-k}$ cancels the Jacobian $r^k$:
\[
\bigl\|((S_{V,w})_*\mu_{V_w})*K_B^V\bigr\|_1
=\int_{V_w}r^{-k}
\bigl|\mu_{V_w}*K_{Br}^{V_w}(z)\bigr|r^k\dd z
=U_{Br}(V_w).
\]
This proves the inequality in \eqref{eq:recursion}.
For the equality in \eqref{eq:recursion}, by iteration we get
\[
(P^nF)(V)=\sum_{i_1,\ldots,i_n}p_{i_1}\cdots p_{i_n}
F(O_{i_n}^{-1}\cdots O_{i_1}^{-1}V)
=\sum_{|w|=n}p_w F(O_w^{-1}V).
\]

We now apply the pointwise mixing estimate \eqref{eq:grassmix}.
Set $R=B\rho^n\ge1$. Applying \eqref{eq:grassmix} to $U_R$ we get
\[
P^nU_R(V)
\le \int_{\Gr(d,k)}U_R(W)\dd\sigma_k(W)
+Ce^{-cn^{1/3}}
\bigl(\|U_R\|_\infty+\Lip(U_R)\bigr).
\]
Combining this with \eqref{eq:recursion} and the estimates
\eqref{eq:three} we obtain
\begin{equation}\label{eq:master}
U_B(V)\le CR^{-\alpha}+CR e^{-c n^{1/3}}.
\end{equation}

We choose $n$ so that both terms on the right-hand side of \eqref{eq:master} are bounded by the same order. Since $R^{-\alpha}=Re^{-cn^{1/3}}$ is equivalent to $(1+\alpha)\log R=cn^{1/3}$, and $\log B=\chi n+\log R$, this leads to make $\log R$ of order $(\log B)^{1/3}$. Shortly denote by
\[
T=\log B,\qquad
\gamma=c(2\chi)^{-1/3},\qquad
a=\frac{\gamma}{1+\alpha},
\]
where $c$ is the constant in \eqref{eq:master}. For $T\ge T_0:=\max\{\log2,(4a)^{3/2},4\chi\}$, we can take
\[
n=\left\lfloor\frac{T-aT^{1/3}}{\chi}\right\rfloor,
\]
so that $n\ge T/(2\chi)$ and
\[
aT^{1/3}\le\log R<aT^{1/3}+\chi,
\]
which also implies
\[
R^{-\alpha}\le e^{-\alpha aT^{1/3}} \text{ and } Re^{-cn^{1/3}}
\le e^\chi e^{-(\gamma-a)T^{1/3}}=e^\chi e^{-\alpha aT^{1/3}}.
\]
We may therefore take
\[
c'=\alpha a=\frac{t-k}{2}\,a>0.
\]
This provides the right bound $U_B(V)\le C'e^{-c'T^{1/3}}$ whenever $B\ge e^{T_0}$. For $1\le B <e^{T_0}$, the same estimate applies by enlarging $C'$, since $U_B(V)\leq 2$. Hence
\begin{equation}\label{eq:increment}
\sup_{V\in\Gr(d,k)} U_B(V)\leq C'e^{-c'(\log B)^{1/3}},\qquad B\geq 1.
\end{equation}

\subsection{Densities and their uniform tails}

We can now finish the proof of Theorem \ref{thm:energy} in the case of equal contraction ratios.

\subsubsection*{Construction of the $L^1$ limits}
For $V\in\Gr(d,k)$ define the sequence of Gaussian smoothings:
\[
g_{V,j}(y)
:=(\mu_V*\Gamma_{2^{-j}}^V)(y)
=\int_V\Gamma_{2^{-j}}^V(y-x)\dd\mu_V(x), \quad j \geq 0.
\]
These smoothings are defined without assuming that $\mu_V$
is absolutely continuous. Each $g_{V,j}$ is a nonnegative
Borel function, and by Fubini's theorem
\[
\int_Vg_{V,j}\dd\LL_V^k
=\int_V\left(\int_V\Gamma_{2^{-j}}^V(y-x)\dd y\right)
\dd\mu_V(x)=1.
\]
By \eqref{eq:increment}, there are fixed $C_0,c_0>0$,
independent of $V$, such that for every $j\ge1$,
\[
\|g_{V,j}-g_{V,j-1}\|_{L^1(V)}
=\|\mu_V*K_{2^j}^V\|_{L^1(V)}
=U_{2^j}(V)\le C_0e^{-c_0j^{1/3}}.
\]
For $j'>j\ge0$, telescoping and the triangle inequality implies
\begin{equation}\label{eq:uniform-cauchy}
\|g_{V,j'}-g_{V,j}\|_{L^1(V)}
\le C_0\sum_{\ell=j+1}^{j'} e^{-c_0\ell^{1/3}}
\le E_j,\qquad
E_j:=C_0\sum_{\ell>j}e^{-c_0\ell^{1/3}}.
\end{equation}
Note that the bound $E_j$ is \textit{independent} of $V$. To obtain both convergence
and a quantitative rate, we estimate this series
tail. Since the summand is decreasing, by using the change of variables $u=x^{1/3}$,
\begin{align*}
\sum_{\ell>j}e^{-c_0\ell^{1/3}}
\le\int_j^\infty e^{-c_0x^{1/3}}\dd x
=3e^{-c_0j^{1/3}}
\left(\frac{j^{2/3}}{c_0}
+\frac{2j^{1/3}}{c_0^2}
+\frac{2}{c_0^3}\right).
\end{align*}
Take $c_1=c_0/2$. The function
$(1+u+u^2)e^{-c_0u/2}$ is bounded on $[0,\infty)$,
so the last right-hand side of the previous inequality is bounded by $C_1e^{-c_1j^{1/3}}$ for a
suitable constant $C_1$.
Thus for constants $C_1,c_1>0$ independent of $V,j$,
\[
E_j\le C_1e^{-c_1j^{1/3}}\rightarrow0.
\]
Fix $V \in \Gr(d,k)$. Since $L^1(V)$ is complete and by \eqref{eq:uniform-cauchy} the sequence is Cauchy, there exists a limit $f_V\in L^1(V)$. Letting $j'\to\infty$ in that inequality, using continuity of
the $L^1$ norm, we have
$\|f_V-g_{V,j}\|_{L^1(V)}\le E_j$.
Because this bound holds for every $V$ with the same $E_j$,
\begin{equation}\label{eq:uniform-limit}
\sup_{V\in\Gr(d,k)}
\|f_V-g_{V,j}\|_{L^1(V)}
\le E_j\le C_1e^{-c_1j^{1/3}}.
\end{equation}
Since $g_{V,j}\ge0$, we have for $(f_V)^-(y)=\max\{-f_V(y),0\}$ that
$\int_V(f_V)^-\dd y\le\|f_V-g_{V,j}\|_1\to0$,
so $f_V\ge0$ almost everywhere. Also
\[
\left|\int_Vf_V\dd y-1\right|
\le\|f_V-g_{V,j}\|_1\rightarrow0.
\]
Hence $f_V\LL_V^k$ is a probability measure for every $V$.

\subsubsection*{Identification of the limit with $\mu_V$}
Now we show that $f_V$ is a density of the original
projected measure $\mu_V = (\pi_V)_* \mu$. Fix $V$, and let $\varphi:V\to\R$ be bounded and continuous.
By Fubini's theorem we can write
\begin{equation}\label{eq:gaussian-test}
\int_V\varphi(y)g_{V,j}(y)\dd y
=\int_V\int_V\varphi(x+2^{-j}z)
\Gamma_1^V(z)\dd z\dd\mu_V(x).
\end{equation}
Fubini applies because the absolute integral is bounded by
$\|\varphi\|_\infty$: both $\mu_V$ and
$\Gamma_1^V\LL_V^k$ are probability measures.
For every fixed $x,z\in V$,
$\varphi(x+2^{-j}z)\to\varphi(x)$ as $j\to\infty$.
Dominated convergence on this product probability space therefore
shows that the right-hand side of \eqref{eq:gaussian-test}
converges to $\int_V\varphi\dd\mu_V$.
Thus $g_{V,j}\LL_V^k$ converges weakly to $\mu_V$.
On the other hand, \eqref{eq:uniform-limit} implies
\[
\left|\int_V\varphi g_{V,j}\dd y-\int_V\varphi f_V\dd y\right|
\le\|\varphi\|_\infty E_j\rightarrow0.
\]
The two limits agree for every bounded continuous
$\varphi$ so the measures $\mu_V=f_V\LL_V^k.$ Since $V$ was arbitrary, this identifies every constructed
limit and proves \eqref{eq:main-rate}. We finally just choose a Borel representative of \(f_V\) which is nonnegative, finite, and zero outside \(\pi_V A\).

\subsubsection*{A uniform estimate for the density tails}
We now use \eqref{eq:uniform-limit} to bound the mass of the set where $f_V$ is large to get the bound for density. Firstly, for every $V,j$ we have
\[
\|g_{V,j}\|_\infty
\le\|\Gamma_{2^{-j}}^V\|_\infty
=(2\pi)^{-k/2}2^{kj}\le2^{kj}.
\]
For $M\ge M_*:=2^{2(k+1)}$, choose
\[
j=j(M):=\left\lfloor\frac{\log_2(M/2)}{k}\right\rfloor.
\]
This choice depends only on $M$ and $k$ and satisfies
\[
j\ge1,\qquad 2^{kj}\le M/2,\qquad
j\ge\frac{\log M}{2k\log2}.
\]
On the measurable set $\{f_V>M\}$, we have almost everywhere the bound:
\[
f_V-g_{V,j}\ge f_V-M/2\ge f_V/2.
\]
Integrating and using \eqref{eq:uniform-limit} we obtain
\begin{align*}
\int_{\{f_V>M\}}f_V\dd\LL_V^k
&\le2\|f_V-g_{V,j(M)}\|_{L^1(V)}\\
&\le2C_1e^{-c_1j(M)^{1/3}}
\le C_2e^{-c_2(\log M)^{1/3}}.
\end{align*}
All constants are independent of $V$. For $1\le M<M_*$,
the left-hand side is at most one, so enlarging $C_2$
allows us to use the same estimate to this range as it is bounded, and the bound is still uniform in $V$. This proves \eqref{eq:main-tail}.

\section{Absolute continuity for unequal contraction ratios}\label{sec:unequal}

We now consider the IFS \eqref{eq:ifs} with arbitrary contraction
ratios $0<r_i<1$. Under the finite-energy
hypothesis, we prove the absolute-continuity and density bounds in
Theorem~\ref{thm:energy}. Recall that the estimates for $U_R(V)=\|\mu_V*(\Gamma_{1/R}^V-\Gamma_{2/R}^V)\|_{L^1(V)}$ in Lemma~\ref{lem:observable}
use only compact support and finite energy. Hence we still have
\[
 0\le U_R\le2,\qquad \Lip(U_R)\le CR,\qquad
 \sigma_k(U_R)\le CR^{-\alpha}\quad(R\ge1),
\]
where $k<t<\min(t_0,k+2)$ and $\alpha=(t-k)/2>0$.
The only new analytic step is to recover \eqref{eq:increment} when \(R=Br_w\) depends on the word \(w\).

\subsection{Dense rotations among words with fixed symbol counts}
We first find a way to pick up a subset of finite words corresponding to a common
contraction ratio without losing dense rotations. For this we will use Breuillard-Gelander
\cite[Corollary~2.5]{BreuillardGelander03} again.

For a word $u$ over $\{1,\ldots,m\}$, let
$\mathbf c(u)=(c_i(u))_{i=1}^m\in\N_0^m$ be the \textit{count vector}, whose $i$th entry $c_i(u)$ is the number of occurrences of $i$ in the word $u$. If $g_1,\ldots,g_m\in G$, write $g_u$ for their ordered product along $u$. For the empty word, we set $g_{\varnothing}=e$ and $\mathbf c(\varnothing)=0$.

\begin{lemma}\label{lem:balanced}
Let $G$ be a compact connected semisimple Lie group, and
suppose $\langle g_1,\ldots,g_m\rangle$ is dense in $G$.
There is a vector $\mathbf C\in\N^m$ such that, for every
integer vector $\mathbf D\ge\mathbf C$ coordinatewise,
\[
 \overline{\langle g_u:\mathbf c(u)=\mathbf D\rangle}=G.
\]
\end{lemma}
\begin{proof}
Similar to the proof of Lemma~\ref{lem:block}, define the subgroup
\[
 B:=\langle g_ug_v^{-1}:\mathbf c(u)=\mathbf c(v)\rangle.
\]
Let $H:=\overline{B}$. If $\mathbf c(u)=\mathbf c(v)$, then also $\mathbf c(iu)=\mathbf c(iv)$. Hence, for every generator $g_ug_v^{-1}$ of $B$,
\[
g_i(g_ug_v^{-1})g_i^{-1}=g_{iu}g_{iv}^{-1}\in B.
\]
Thus $g_iBg_i^{-1}\subset B$, and taking closures implies
$g_iHg_i^{-1}\subset H$. As in the proof of Lemma~\ref{lem:block}, we can use compactness to find integers $q_j\to\infty$ such that $g_i^{q_j}\to e$. Thus $g_i^{q_j-1}\to g_i^{-1}$ implies
$g_i^{-1}Hg_i\subset H$. Hence $g_iHg_i^{-1}=H$. Since
$\langle g_1,\ldots,g_m\rangle$ is dense and $H$ is closed, $H$ is normal.

Since $H$ is normal in $G$, the quotient $G/H$ is a group. The words $ij$ and $ji$ have the same count vector, and therefore
\[
[g_i,g_j]=g_{ij}g_{ji}^{-1}\in H.
\]
Thus the images of the $g_i$ in $G/H$ commute. Since the $g_i$ generate a dense subgroup of $G$, their images generate a dense subgroup of $G/H$. Hence $G/H$ is abelian, which implies that the commutator subgroup satisfies $[G,G]\subset H$. Since $G$ is connected and semisimple, $\overline{[G,G]}=G$. As $H$ is closed, it follows that $H=G$.

Now apply \cite[Corollary~2.5]{BreuillardGelander03} to $B$ and choose a finite collection
$h_1,\ldots,h_r\in B$ generating a dense subgroup of $G$.
Each $h_j$ is a finite product of the generating element of $B$ and their
inverses. Collect all the elements that occur in these products then there are
finitely many pairs $(u_\ell,v_\ell)$, $1\leq\ell\leq L$, with
$\mathbf c(u_\ell)=\mathbf c(v_\ell)$, such that
\[
 \overline{\langle g_{u_\ell}g_{v_\ell}^{-1}:
 1\leq\ell\leq L\rangle}=G.
\]
Indeed, their generated subgroup contains every $h_j$.
Choose $\mathbf C=(C_i)_{i=1}^m\in\N^m$ such that $C_i \ge c_i(u_\ell)=c_i(v_\ell)$ for all $i = 1,\dots,m$ and for every $\ell$.
Now for any given integer vector $\mathbf D\geq\mathbf C$, for each $\ell$ we may find a word $w_\ell$ (possibly empty word) with $\mathbf c(w_\ell)=\mathbf D-\mathbf c(u_\ell)$. Then
\[
\mathbf c(u_\ell w_\ell)=\mathbf c(v_\ell w_\ell)=\mathbf D,
\qquad
g_{u_\ell w_\ell}g_{v_\ell w_\ell}^{-1}
=
g_{u_\ell}g_{v_\ell}^{-1}.
\]
Thus $\langle g_u:\mathbf c(u)=\mathbf D\rangle$ contains
$g_{u_\ell}g_{v_\ell}^{-1}$ for every $\ell$. Since these elements
generate a dense subgroup of $G$, so does
$\langle g_u:\mathbf c(u)=\mathbf D\rangle$.
\end{proof}

\subsection{Blocks with fixed symbol counts}
We use the symbolic probability space from Section~\ref{subsec:notation}.
By a Bernoulli sequence we mean $\omega=(\omega_j)_{j\ge1}$ with
law $\Pp$, so its symbols are independent with probabilities $(p_i)_{i=1}^m$.
A Bernoulli word of length $n$ is a finite block of $n$ such symbols,
it equals $w\in\mathcal I^n$ with probability $p_w$.

By Lemma~\ref{lem:balanced}, there is
$\mathbf C\in\N^m$ such that the rotations of words with
count vector $\mathbf C=(C_i)_{i=1}^m$ generate a dense subgroup of $G=\SO(d)$.
Put $\ell=\sum_i C_i$, and define a probability measure on $G$
\[
 \nu=\frac{1}
      {\sum_{\substack{w\in\mathcal I^\ell,\mathbf c(w)=\mathbf C}}p_w}\sum_{\substack{w\in\mathcal I^\ell,\mathbf c(w)=\mathbf C}}p_w\delta_{O_w}.
\]
Thus $\nu$ is the law of the random rotation $O_w$ of an $\ell$-letter
Bernoulli word conditional on having count vector $\mathbf C$.
Since $p_w=\prod_i p_i^{C_i}$ if a fixed product for every word in this class,
the conditional distribution on these words is in fact uniform.
In particular, every such word has positive conditional probability,
so $\supp\nu$ generates a dense subgroup of $G$.
By Lemma~\ref{lem:block} applied to $\nu$, one can find an integer $b\ge1$ with
\[
 \overline{\left\langle
 \supp\bigl(\widetilde{\nu^{*b}}*\nu^{*b}\bigr)
 \right\rangle}=G.
\]
Set $L=b\ell$ and $\mathbf D_*=b\mathbf C=(bC_i)_{i=1}^m$. For each $\mathbf D\in\mathcal D:=
\{\mathbf D\in\N_0^m:\sum_iD_i=L\}$, define
\[
 p_{\mathbf D}=\sum_{\mathbf c(w)=\mathbf D}p_w,\qquad
 \eta_{\mathbf D}=p_{\mathbf D}^{-1}
       \sum_{\mathbf c(w)=\mathbf D}p_w\delta_{O_w},\qquad
 a(\mathbf D)=-\sum_iD_i\log r_i.
\]
Here we have $p_{\mathbf D}>0$, $\sum_{\mathbf D\in\mathcal D}p_{\mathbf D}=1$, and every word $w$ with $\mathbf c(w)=\mathbf D$ has the same contraction ratio $e^{-a(\mathbf D)}$ in the map $S_w$. Also, if $w_1,\ldots,w_b$ all satisfy $\mathbf c(w_j)=\mathbf C$, then their concatenation satisfies $\mathbf c(w_1\cdots w_b)=\mathbf D_*$. Hence $\supp\nu^{*b}\subset\supp\eta_{\mathbf D_*},$
which implies
\begin{equation}\label{eq:count-support}
\overline{\left\langle
\supp\bigl(\widetilde{\eta_{\mathbf D_*}}*\eta_{\mathbf D_*}\bigr)
\right\rangle}=G.
\end{equation}
Thus $\mathbf D_*$ is a count vector for which the conditional rotation law satisfies the dense support condition required by Varj\'u's $L^2$ weak gap theorem. Set $p_*:=p_{\mathbf D_*}>0$, and define
\[
a_-:=\min_{\mathbf D\in\mathcal D}a(\mathbf D)>0,
\qquad
a_+:=\max_{\mathbf D\in\mathcal D}a(\mathbf D)<\infty.
\]
Every length-$L$ block has contraction between $e^{-a_+}$ and $e^{-a_-}$, and a block with count vector $\mathbf D_*$ occurs with probability $p_*$. The other count vectors are kept with their original probabilities.

\subsection{Mixing conditional on the count vectors}
For $\mathbf D\in\mathcal D$, define its associated average operator by
\[
 (P_{\mathbf D}F)(V)
 =\int_G F(O^{-1}V)\dd\eta_{\mathbf D}(O).
\]
The next lemma is the form of Proposition~\ref{prop:mix} needed
when different rotation laws occur between the mixing steps.

\begin{lemma}\label{lem:count-mixing}
There are constants $C,c>0$ such that, for every $n\ge 1$, every sequence
$\mathbf D_1,\ldots,\mathbf D_n\in\mathcal D$ and every
Lipschitz function $F:\Gr(d,k)\to\C$,
\begin{equation}\label{eq:count-mixing}
 \|P_{\mathbf D_1}\cdots P_{\mathbf D_n}F-\sigma_k(F)\|_\infty
 \le Ce^{-cN_*(n)^{1/3}}\bigl(\|F\|_\infty+\Lip(F)\bigr),
\end{equation}
where $N_*(n)=\#\{j\le n:\mathbf D_j=\mathbf D_*\}$.
\end{lemma}
\begin{proof}
First work on $G$, with the average operator $Q_{\mathbf D}H(g)=\int H(O^{-1}g)\dd\eta_{\mathbf D}(O)$.
Each $Q_{\mathbf D}$ preserves Haar means and contracts
the $L^2$ norm, the supremum norm, and the Lipschitz seminorm.
Subtract the mean and rescale so that
$\vartheta(H)=0$ and $\|H\|_\infty+\Lip(H)\le1$.
Define
\[
 H_0=H,\qquad H_j=Q_{\mathbf D_{n-j+1}}H_{j-1}
 \quad(1\le j\le n),
\]
so that $H_n=Q_{\mathbf D_1}\cdots Q_{\mathbf D_n}H$, and the operators act from right to left. Each $H_j$ has mean
zero and satisfies $\|H_j\|_\infty+\Lip(H_j)\le1$.
Fix $0<\varepsilon<1/2$. Similar to the proof of Proposition~\ref{prop:mix}, whenever
$\mathbf D_{n-j+1}=\mathbf D_*$ and
$\|H_{j-1}\|_2>\varepsilon$, passing from $H_{j-1}$ to $H_j$
decreases the $L^2$ norm by a factor at most
\[
 1-\frac{c}{\log^2(2+\varepsilon^{-1})}.
\]
where the constant $c>0$ only depends on the rotation law $\eta_{\mathbf{D}_*}$. This follows from \eqref{eq:count-support} and
Theorem~\ref{thm:varju}, applied to $H_{j-1}/\|H_{j-1}\|_2$,
whose supremum norm plus Lipschitz seminorm is at most
$\varepsilon^{-1}$. Since the average operators do not increase the
$L^2$ norm, if a norm $\|H_{j-1}\|_2$ falls to at most $\varepsilon$ at any
stage, it stays at most $\varepsilon$ after this. Otherwise all the norms up to $n-1$ stay above $\varepsilon$ and the contraction factor can be used at all $N_*(n)$ occurrences of
$Q_{\mathbf D_*}$. This implies that
\[
 \|H_n\|_2
 \le\max\left\{\varepsilon,
 \exp\left(-\frac{cN_*(n)}{\log^2(2+\varepsilon^{-1})}\right)\right\}.
\]
Choose $\gamma>0$ such that $4\gamma^3\le c$, and set
\[
 N_0=\left\lceil\left(\frac{\log3}{\gamma}\right)^3\right\rceil.
\]
For $N_*(n)\ge N_0$, take $\varepsilon=e^{-\gamma N_*(n)^{1/3}}$.
The choice of $N_0$ implies that $\gamma N_*(n)^{1/3}\ge\log3$,
so $\varepsilon\le1/3<1/2$ and
\[
 \log(2+\varepsilon^{-1})
 =\log\bigl(2+e^{\gamma N_*(n)^{1/3}}\bigr)
 \le\log3+\gamma N_*(n)^{1/3}
 \le2\gamma N_*(n)^{1/3}.
\]
Moreover, the second term in the maximum is also at most
$\varepsilon$, since
\[
 \exp\left(-\frac{cN_*(n)}{\log^2(2+\varepsilon^{-1})}\right)
 \le\exp\left(-\frac{c}{4\gamma^2}N_*(n)^{1/3}\right)
 \le e^{-\gamma N_*(n)^{1/3}}=\varepsilon.
\]
Thus $\|H_n\|_2\le e^{-\gamma N_*(n)^{1/3}}$ whenever
$N_*(n)\ge N_0$. Both $\gamma$ and $N_0$ are independent of
$n$ and the type sequence. If $N_*(n)<N_0$, we have
$\|H_n\|_2\le\|H_0\|_2\le1$, regardless of the size of $n$.
Thus by choosing $C=e^{\gamma N_0^{1/3}}$ we have
\[
\|H_n\|_2\le Ce^{-\gamma N_*(n)^{1/3}}
\]
for every $n$ and every count-vector sequence. Since $\Lip(H_n)\le1$, the same ball-volume argument as in the proof of Proposition~\ref{prop:mix} yields the bound:
\[
\|H_n\|_\infty
\le C\|H_n\|_2^{2/(\dim G+2)}
\le C'e^{-c'N_*(n)^{1/3}}
\]
for constants $C',c'>0$ independent of $n$ and count-vector sequences. We can then rescale to get the estimate for general $H$, and finally using $H(g)=F(gV_0)$ as in Section~\ref{sec:mix}, we obtain \eqref{eq:count-mixing} for the Grassmannian operators.
\end{proof}

\subsection{The stopping-time argument}

We now choose the number of blocks depending on their contraction ratios, stopped when the rescaled smoothing parameter reaches the same range as in the equicontractive case.

\begin{proposition}\label{prop:unequal-increment}
Under the hypotheses of Theorem~\ref{thm:energy}, there are constants $C,c>0$ such that
\[
\sup_{V\in\Gr(d,k)}U_B(V)
\le Ce^{-c(\log B)^{1/3}}\qquad(B\ge 1).
\]
\end{proposition}

\begin{proof}
Let $W_j=\omega_{(j-1)L+1}\cdots\omega_{jL}$, $j\ge1$, be the
successive disjoint blocks of length $L$ in a Bernoulli sequence
$\omega$, and let $\mathbf D_j:=\mathbf c(W_j)$.
These blocks are independent Bernoulli words, so the types are
independent with law $(p_{\mathbf D})_{\mathbf{D}\in \mathcal{D}}$.
Conditioned on any finite type sequence $\mathbf{D}_{j_1},\cdots,\mathbf{D}_{j_l}$, the block rotations are independent with respective laws $\eta_{\mathbf D_{j_1}},\cdots,\eta_{\mathbf D_{j_l}}$.

Write $B=e^T$ for convenience. For sufficiently large $T$, we will choose a parameter $u$
satisfying $a_+\le u\le T/2$. The range of $T$ and the choice of $u$ will be made at the end of the proof. Define the stopping-time
\[
 \tau:=\min\left\{n\ge1:
       \sum_{j=1}^n a(\mathbf D_j)\ge T-u\right\}
\]
and write
$$ R=B\prod_{j=1}^{\tau}r_{W_j}
  =\exp\left(T-\sum_{j=1}^{\tau}a(\mathbf D_j)\right).$$
Since $a_-\le a(\mathbf D_j)\le a_+$, we have
\[
 \frac{T-u}{a_+}\le\tau
 \le\left\lceil\frac{T-u}{a_-}\right\rceil.
\]
Moreover, by minimality of $\tau$,
\[
 T-u\le\sum_{j=1}^{\tau}a(\mathbf D_j)<T-u+a_+,
\]
and hence
\begin{equation}\label{eq:stopped-scale}
 e^{u-a_+}\le R\le e^u.
\end{equation}
In particular, $R\ge1$. Both $\tau$ and $R$ depend only on
the block types.

We next show that, with high probability, at least $c_0T$ of the stopped blocks have count vector $\mathbf D_*$, that is, are of type $\mathbf D_*$ for some constant $c_0>0$ that we will choose soon. Let
\[
 n_0=\left\lfloor\frac{T-u}{a_+}\right\rfloor.
\]
Then $n_0\le\tau$, and, since $u\le T/2$, $n_0\ge T/4a_+$ for all sufficiently large $T$. The lower bound $T/4a_+$ is independent of the choice of $u$ in the range $[a_+, T/2]$.

The counting number
\[
 Z:=\#\{j\le n_0:\mathbf D_j=\mathbf D_*\}
\]
has the binomial distribution with parameters $n_0,p_*$ (thus has mean $p_*n_0$), and
\[
 N_*(\tau):=\#\{j\le\tau:\mathbf D_j=\mathbf D_*\}\ge Z
\]
since $\tau\ge n_0$. Using the Chernoff bounds for binomial lower-tails we can get
\[
 \Pp\{Z<p_*n_0/2\}\le e^{-p_*n_0/8}.
\]
Then by choosing $c_0=p_*/(8a_+)$ and using $n_0\ge T/4a_+$ we get
\begin{equation}\label{eq:many-mixing-blocks}
 \Pp\{N_*(\tau)<c_0T\}\le e^{-\frac{c_0}{4}T}.
\end{equation}

The stopped words $w=W_1\cdots W_\tau$ form a finite complete
prefix code: every infinite symbolic sequence has exactly one
such prefix. If we iterate the self-similarity $\mu=\sum_i p_i(S_i)_*\mu$ along this finite tree, we have the identity:
\[
 \mu=\sum_{w\text{ stopped}}p_w(S_w)_*\mu.
\]
Applying the triangle inequality and the Gaussian change of variables
used in deriving \eqref{eq:recursion}, with $r_w$ in place of $\rho^n$, we obtain
\begin{equation}\label{eq:stopped-recursion}
 U_B(V)\le\sum_{w\text{ stopped}}p_wU_{Br_w}(O_w^{-1}V)
 =\E\,U_R(O_{W_1\cdots W_\tau}^{-1}V).
\end{equation}
Note that here $R$ and $\tau$ are random. To estimate this expectation on the right-hand side, fix a type sequence
$(\mathbf d_1,\ldots,\mathbf d_n)$ for which
\[
 \sum_{j=1}^{n-1}a(\mathbf d_j)<T-u
 \le\sum_{j=1}^{n}a(\mathbf d_j).
\]
On the event $\{\mathbf D_j=\mathbf d_j,\ 1\le j\le n\}$,
the stopping time is necessarily $n$, and $R$ is fixed.
Conditioning additionally on $\tau=n$ therefore imposes no further
restriction. The block rotations still have their independent conditional
laws, so
\[
 \E\!\left[
 U_R(O_{W_1\cdots W_\tau}^{-1}V)
 \,\middle|\,
 \tau=n,\ \mathbf D_j=\mathbf d_j,\ 1\le j\le n
 \right]=(P_{\mathbf d_1}\cdots P_{\mathbf d_n}U_R)(V).
\]
Here the operator order agrees with
$O_{W_1\cdots W_n}^{-1}=O_{W_n}^{-1}\cdots O_{W_1}^{-1}$.

Let $c_1>0$ denote the exponential constant in
Lemma~\ref{lem:count-mixing}. On the event $N_*(\tau)\ge c_0T$, by
lemma~\ref{lem:count-mixing}, \eqref{eq:stopped-scale}, and the three estimates for $U_R$, there is a constant $C>0$ such that
\begin{align*}
 (P_{\mathbf D_1}\cdots P_{\mathbf D_\tau}U_R)(V)
 &\le CR^{-\alpha}
      +CR e^{-c_1(c_0T)^{1/3}}\\
 &\le Ce^{\alpha a_+}e^{-\alpha u}+Ce^{u-c_2T^{1/3}},
\end{align*}
where $c_2=c_1c_0^{1/3}>0$.
These constants are independent of $T$, $u$, and the stopped type
sequence. On the complementary event, the conditional expectation
is at most $2$. Averaging in \eqref{eq:stopped-recursion} and
absorbing $e^{\alpha a_+}$ into $C$, as well as using \eqref{eq:many-mixing-blocks}, we obtain
\[
 \sup_VU_{e^T}(V)
 \le Ce^{-\alpha u}+Ce^{u-c_2T^{1/3}}+2e^{-\frac{c_0}{4}T}.
\]

Finally, choose
\[
 \gamma=\frac{c_2}{1+\alpha}>0,\qquad u=\gamma T^{1/3}.
\]
This satisfies $a_+\le u\le T/2$ for all $T$ such that $a_+\le\gamma T^{1/3}\le T/2$, that is
\[
T\ge T_0:= \max\{(2\gamma)^{3/2},(a_+/\gamma)^3\}.
\]
Also we have
\[
c:=c_2-\gamma=\alpha\gamma>0.
\]
Therefor the first two terms have the same decay
$e^{-c T^{1/3}}$, while the third decays faster. Hence
\[
 \sup_VU_B(V)\le (2C+2)e^{-c(\log B)^{1/3}}
\]
for all $B\ge e^{T_0}$ such that $e^{-c(\log B)^{1/3}}\ge e^{-(c_0/4)\log B}$, that is $B\ge \max\{e^{T_0}, e^{(4c/c_0)^{3/2}}\}$. Increasing $C$, using $U_B\le2$,
covers the remaining range in $B\ge1$.
\end{proof}

\subsection{Conclusion of the absolute-continuity argument}
Proposition~\ref{prop:unequal-increment} provides precisely the
smoothing estimate used in Section~\ref{sec:smoothing}.
The subsequent construction of the densities, the approximation
bound \eqref{eq:main-rate}, and the tail bound \eqref{eq:main-tail}
use no further assumption on the contraction ratios, hence remain valid.
This proves the absolute-continuity and density bounds in
Theorem~\ref{thm:energy}.

\section{Proofs of Corollary \ref{cor:bound} and Corollary \ref{cor:entropy}}\label{sec:entropy}

Assume the assumptions in Theorem \ref{thm:energy} hold. For $V\in\Gr(d,k)$ let $f_V$ be the density function.

\subsection{Uniform logarithmic integrability} To obtain Corollary~\ref{cor:bound}, we only need to use the tail estimate \eqref{eq:main-tail}. For $q>0$, write
\[
(\log^+f_V(y))^q
=\int_0^{\log^+f_V(y)}q u^{q-1}\dd u.
\]
By using \eqref{eq:main-tail} this implies
\begin{align*}
\int_V f_V(\log^+f_V)^q\dd\LL_V^k
&=q\int_0^\infty u^{q-1}\mu_V\{f_V>e^u\}\dd u\\
&\le Cq\int_{0}^\infty u^{q-1}e^{-c_2u^{1/3}}\dd u
<\infty.
\end{align*}
This bound is independent of $V$, and it implies the uniform integrability.

\subsection{Uniform entropy bound}
We use the cube $J$, the normalized Lebesgue measure
$\lambda_J$, and the dyadic partitions introduced before
Corollary~\ref{cor:entropy}. The relative density is $q_V=|J|f_V$.
For $y\in D\in\mathcal D_n(J)$ define
\[
 Z_{V,n}(y)=\frac{\mu_V(D)}{\lambda_J(D)}
           =2^{kn}\mu_V(D).
\]
It equals to the $\lambda_J$-average of $q_V$ on $D$. Since $\lambda_J(D)=2^{-kn}$, we have
\begin{align*}
 \int_J Z_{V,n}\log Z_{V,n}\dd\lambda_J
 &=\sum_{D\in\mathcal D_n(J)}\mu_V(D)
                         \log(2^{kn}\mu_V(D))\\
 &=kn\log2-H(\mu_V,\mathcal D_n(J)),
\end{align*}
where a cell $D$ of zero mass contributes zero in the sum. Then by using the convexity of $u\mapsto u\log u$ we get
\begin{align*}
 0\le\int_J Z_{V,n}\log Z_{V,n}\dd\lambda_J
 &\le\int_J q_V\log q_V\dd\lambda_J\\
 &=\log|J|+\int_J f_V\log f_V\dd y\\
 &\le\log|J|+\int_J f_V\log^+f_V\dd y:= C<\infty.
\end{align*}
The last bound comes from Corollary~\ref{cor:bound} with $q=1$.

\section{Dimension conservation for measures under strong separation}\label{sec:conditional}
We now work under the general IFS \eqref{eq:ifs} satisfying \SSC, and suppose that every projected measure $\mu_V$ has a
density $f_V$ satisfying \eqref{eq:main-tail}, as proved in
Sections~\ref{sec:smoothing} and~\ref{sec:unequal} under the hypotheses
of Theorem~\ref{thm:energy}. Recall the notation $h(p)$ and $\chi(p)$, and denote by $s_\mu=h(p)/\chi(p)$ the exact dimension of $\mu$.

\subsection{The conditional-cylinder formula}\label{subsec:conditional-cylinders}
Fix $V$ and a disintegration as in \eqref{eq:disintegration},
with the conditional measures defined on a Borel set $B_0\subset V$
of full $\mu_V$-measure and supported on
$A\cap\pi_V^{-1}\{y\}$ for $y\in B_0$.
For a finite word $w$, recall the notations $A_w=S_wA$ and $V_w=O_w^{-1}V$.
As in Section~\ref{subsec:self-similarity-mixing}, define the similarity
$S_{V,w}:V_w\to V$ by
\[
 S_{V,w}(z)=\pi_Va_w+r_wO_wz,\qquad z\in V_w,
\]
and let
\[
 z_w(y)=S_{V,w}^{-1}(y)
       =r_w^{-1}O_w^{-1}(y-\pi_Va_w)\in V_w.
\]
The following lemma is the multidimensional form of the identity
used in \cite[Section~5]{Rapaport20}.
\begin{lemma}\label{lem:conditional}
For $\mu_V$-almost every $y$, simultaneously for all finite words $w$,
\begin{equation}\label{eq:conditional}
 \mu_{V,y}(A_w)=
 p_wr_w^{-k}\frac{f_{V_w}(z_w(y))}{f_V(y)}.
\end{equation}
\end{lemma}
\begin{proof}
Recall that we have, by self-similarity and \SSC,
$\mu|_{A_w}=p_w(S_w)_*\mu$.
By the change of variables on the two $k$-planes, the density of
$(\pi_V)_*(\mu|_{A_w})$ is
\[
 y\longmapsto p_wr_w^{-k} f_{V_w}(z_w(y)).
\]
For every Borel $B\subset V$, we may write
\[
 \mu(A_w\cap\pi_V^{-1}B)
 =\int_{B\cap B_0}\mu_{V,y}(A_w)f_V(y)\dd\LL_V^k(y).
\]
Thus we may deduce from the uniqueness of Radon-Nikodym derivatives that
\[
 \mu_{V,y}(A_w)f_V(y)=p_wr_w^{-k}f_{V_w}(z_w(y))
 \quad\text{for }\LL_V^k\text{-almost every }y\in B_0.
\]
Since $f_V$ is a probability density, its zero set has
$\mu_V$-measure zero and it is finite $\mu_V$-almost
everywhere. We may therefore divide by $f_V(y)$ on a set
of full $\mu_V$-measure.
By removing the countable union of the exceptional null sets we verify
the simultaneously in finite words claim. Take the sum over
the level-$n$ partition of $A$ we also have
\[
 \sum_{|w|=n}p_wr_w^{-k}f_{V_w}(z_w(y))=f_V(y)
 \quad\text{for }\mu_V\text{-almost every }y.
\]
Thus the conditional cylinder masses sum to one at every level
on the same set of full $\mu_V$-measure.
\end{proof}

\subsection{Sublinear growth of the density term}

Let us now study the growth of the density term in \eqref{eq:conditional} by realising it as a random variable. Fix $V\in\Gr(d,k)$. We work on the Bernoulli space
$(\Omega,\Pp)=(\mathcal I^{\N},p^{\N})$ from
Section~\ref{sec:prelim}, and write
$\omega=(\omega_1,\omega_2,\ldots)$ for a random sequence.
For $n\geq1$, let
\[
w_n=\omega|_n,\qquad
V_n=O_{w_n}^{-1}V,\qquad
X_n=\Pi(\sigma^n\omega).
\]
Thus $w_n$ is the first $n$ symbols, $V_n$ is the corresponding
rotated projection plane, and $X_n$ is the point coded by the
remaining symbols. In particular,
$\Pi(\omega)=S_{w_n}(X_n)$. The density term that appears in the conditional-cylinder formula above \eqref{eq:conditional} is the random variable:
\[
Y_n=f_{V_n}\bigl(\pi_{V_n}X_n\bigr),
\]
Here note that $Y_n$ is measurable since for fixed $V$, only
countably many planes $V_n$ can occur, so we choose a Borel
representative of each corresponding density, equal to zero outside
its projected attractor.

We show that $\log Y_n=o(n)$ almost surely:

\begin{lemma}\label{lem:sublinear}
For every fixed $V$,
\begin{equation}\label{eq:sublinear}
 \lim_{n\to\infty}\frac{\log Y_n}{n}=0
 \quad \Pp\text{-almost surely}.
\end{equation}
\end{lemma}
\begin{proof}
The prefix $w_n$ depends only on the first $n$ coordinates,
whereas $X_n$ depends only on the remaining coordinates.
Hence $X_n$ is independent of $w_n$ and has law $\mu$.
Conditional on $w_n=w$, the plane $V_n=V_w$ is fixed and
$X_n$ still has law $\mu$.
In particular $0<Y_n<\infty$ almost surely for each $n$,
because any density is positive and finite almost everywhere
with respect to the measure it defines. By countability this
holds simultaneously for all $n$.
Fix $\varepsilon>0$. For every $n\ge 0$ and every word $w$ of length $n$, by \eqref{eq:main-tail},
\[
 \Pp(Y_n>e^{\varepsilon n}\mid w_n=w)
 =\int_{\{f_{V_w}>e^{\varepsilon n}\}}f_{V_w}\dd\LL_{V_w}^k
 \le C e^{-c(\varepsilon n)^{1/3}}.
\]
For the lower tail, write $v_k=\LL^k(B_{\R^k}(0,1))$ and
recall that $A\subset B(0,D_A)$.
Thus $\pi_{V_w}A\subset B_{V_w}(0,D_A)$, a ball of volume
$v_kD_A^k$, independently of the prefix. Therefore we have:
\[
 \Pp(Y_n<e^{-\varepsilon n}\mid w_n=w)
 =\int_{\pi_{V_w}A\cap\{f_{V_w}<e^{-\varepsilon n}\}}
          f_{V_w}\dd\LL_{V_w}^k
 \le v_kD_A^k e^{-\varepsilon n}.
\]
By the law of total probability we have
\[
 \Pp(Y_n>e^{\varepsilon n})
 =\sum_{|w|=n}p_w\Pp(Y_n>e^{\varepsilon n}\mid w_n=w)
 \le Ce^{-c(\varepsilon n)^{1/3}}.
\]
The same argument implies the unconditional lower-tail bound.
Both are summable. Hence by using the first Borel-Cantelli lemma we deduce that
$e^{-\varepsilon n}\le Y_n\le e^{\varepsilon n}$
eventually almost surely. Taking a countable
$\varepsilon$ tending to zero we obtain \eqref{eq:sublinear}.
\end{proof}

\subsection{From cylinder masses to local dimensions}

We first recall the standard comparison between cylinders and balls under \SSC.

\begin{lemma}\label{lem:balls}
Suppose $m\geq2$ and \SSC\ holds. Set
\[
\Delta=\min_{i\neq j}\operatorname{dist}(S_iA,S_jA)>0.
\]
For $x=\Pi\omega$, let $A_n(x)=A_{\omega|n}$ and
$r_n(x)=r_{\omega|n}$. Then
\begin{equation}\label{eq:ball-cyl}
B(x,(\Delta/2)r_n(x))\cap A
\subset A_n(x)\subset B(x,2D_A r_n(x)).
\end{equation}
If $\nu$ is a probability measure supported on $A$ and
\[
\lim_{n\to\infty}\frac{\log\nu(A_n(x))}{\log r_n(x)}=a<\infty,
\]
then
\[
\lim_{r\downarrow0}\frac{\log\nu(B(x,r))}{\log r}=a.
\]
\end{lemma}

\begin{proof}
The inclusions in \eqref{eq:ball-cyl} follow from the separation
argument in \cite[proof of Proposition~9.7]{Falconer14}. The
conclusion for local dimensions is the usual cylinder-ball
comparison, see \cite[proof of Lemma~17.5]{Falconer14}.
\end{proof}

Fix $V\in\Gr(d,k)$. Write $x=\Pi(\omega)$ and $y=\pi_Vx$.
Recall that
\[
x=S_{w_n}(X_n),\qquad
z_{w_n}(y)=\pi_{V_n}X_n.
\]
Hence Lemma~\ref{lem:conditional} implies for $\Pp$-almost every
$\omega$,
\[
\log\mu_{V,y}(A_{w_n})
=
\log p_{w_n}-k\log r_{w_n}
+\log Y_n-\log f_V(y).
\]
By the strong law of large numbers,
\[
\frac{\log p_{w_n}}{n}\to-h(p),
\qquad
\frac{\log r_{w_n}}{n}\to-\chi(p)
\quad\text{almost surely}.
\]
By Lemma~\ref{lem:sublinear}, $\log Y_n=o(n)$ almost surely, while
$\log f_V(y)$ is finite and does not depend on $n$. Therefore
\begin{equation}\label{eq:cyl-dim}
\frac{\log\mu_{V,\pi_Vx}(A_n(x))}{\log r_n(x)}
\to
\frac{h(p)}{\chi(p)}-k=s_\mu-k
\end{equation}
for $\mu$-almost every $x$.

We now pass from this $\mu$-almost everywhere statement to the
conditional measures. Let $B_0\subset V$ be the full $\mu_V$-measure
set on which the disintegration $y\mapsto\mu_{V,y}$ is defined.
Given
$x\in A\cap\pi_V^{-1}(B_0)$ define
\[
M_n(x)
=
\sum_{|w|=n}
\ind_{A_w}(x)\mu_{V,\pi_Vx}(A_w).
\]
Under \SSC, exactly one level-$n$ cylinder contains $x$, so $ M_n(x)=\mu_{V,\pi_Vx}(A_n(x)).$
Similarly,
$
r_n(x)=\sum_{|w|=n}\ind_{A_w}(x)r_w.
$
Both functions are Borel. Hence the set
\[
E_V:=
\left\{
x\in A\cap\pi_V^{-1}(B_0):
\lim_{n\to\infty}\frac{\log M_n(x)}{\log r_n(x)}
= s_\mu-k
\right\}
\]
is Borel measurable. By \eqref{eq:cyl-dim}, $\mu(E_V)=1$. By
\[
1=\mu(E_V)
=
\int_{B_0}\mu_{V,y}(E_V)\,d\mu_V(y),
\]
we get $\mu_{V,y}(E_V)=1$ for $\mu_V$-almost every $y$. Fix such a $y$. Then for $\mu_{V,y}$-almost every $x$,
\[
\lim_{n\to\infty}
\frac{\log\mu_{V,y}(A_n(x))}{\log r_n(x)}
=s_\mu-k.
\]
Applying Lemma~\ref{lem:balls} to the measure $\nu=\mu_{V,y}$ implies
\[
\lim_{r\downarrow0}
\frac{\log\mu_{V,y}(B(x,r))}{\log r}
=s_\mu-k
\quad\text{for }\mu_{V,y}\text{-almost every }x.
\]
Thus $\mu_{V,y}$ is exact dimensional of dimension $s_\mu-k$ for
$\mu_V$-almost every $y$.

\subsection{Dimension conservation for measures}

Under \SSC, the self-similar measure $\mu$ is exact dimensional and
\[
\dimH\mu=\frac{h(p)}{\chi(p)}=s_\mu,
\]
see \cite[Section~17.3]{Falconer14}. Since
$\mu_V=f_V\,\mathcal L_V^k$, the Lebesgue differentiation theorem implies
\[
\mu_V(B_V(y,r))
\sim f_V(y)v_kr^k
\quad\text{as }r\downarrow0
\]
for $\mu_V$-almost every $y$, with $0<f_V(y)<\infty$. Hence
$\dimH\mu_V=k$. Together with the conditional dimension formula we get
\[
\dimH\mu
=
k+(s_\mu-k)
=
\dimH\mu_V+\dimH\mu_{V,y}
\]
for $\mu_V$-almost every $y$. Thus we have proved \eqref{eq:dc-measure}.

\section{Equivalence and positive projection volumes}\label{sec:equivalence}\label{sec:volumes}
We adapt \cite[Section~7]{Rapaport20} to show that uniform
integrability and self-similarity imply equivalence of $\mu_V$
with Lebesgue measure restricted to $\pi_VA$. This allows us to move statements holding
$\mu_V$-almost everywhere to Lebesgue-almost every
point of $\pi_VA$.

The following consequence of the Lebesgue density theorem
\cite[Corollary~2.14(1)]{Mattila95} will be used in
Proposition~\ref{prop:equivalence} and the proof of Theorem~\ref{thm:sets}.
\begin{lemma}\label{lem:density}
Suppose $\diam A>0$.
Fix $V\in\Gr(d,k)$, and let $E\subset\pi_VA$ be Borel.
If there is a constant $a>0$ such that
\[
 \LL_V^k\bigl((\pi_VS_wA)\setminus E\bigr)\ge a r_w^k
\]
for every finite word $w$, then $\LL_V^k(E)=0$.
\end{lemma}
\begin{proof}
Fix $y\in E$ and choose $x\in A$ with $\pi_Vx=y$.
Choose a coding $\omega=(\omega_1,\omega_2,\ldots)$ of $x$, and let
$w_n=\omega_1\cdots\omega_n$ be its first $n$ symbols. Since $x\in S_{w_n}A$ and orthogonal projection is $1$-Lipschitz,
\[
 \pi_VS_{w_n}A\subset B_V(y,2D_A r_{w_n}).
\]
Recall that $v_k=\LL^k(B_{\R^k}(0,1))$. From the assumption we have:
\[
 \frac{\LL_V^k(B_V(y,2D_A r_{w_n})\setminus E)}
      {\LL_V^k(B_V(y,2D_A r_{w_n}))}
 \ge\frac{a}{v_k(2D_A)^k}>0.
\]
Since $r_{w_n}\to0$, the above bound shows that $y$ is not
a Lebesgue density point of $E$. This holds for every $y\in E$,
whereas the Lebesgue density theorem implies density one at
$\LL_V^k$-almost every point of $E$. Thus $\LL_V^k(E)=0$.
\end{proof}

\begin{proposition}\label{prop:equivalence}
Suppose the projections of a self-similar measure $\mu$ have densities $f_V$ for all
$V\in\Gr(d,k)$, and
\[
 \lim_{M\to\infty}\sup_V\int_{\{f_V>M\}}f_V\dd\LL_V^k=0.
\]
Then $\mu_V\sim\LL_V^k|_{\pi_VA}$ for every $V$.
\end{proposition}
\begin{proof}
By the assumption, we may choose $M_0>0$ such that for every $V\in\Gr(d,k)$,
\[
\int_{\{f_V>M_0\}} f_V\,d\mathcal L_V^k\le \frac12.
\]
For every $V\in\Gr(d,k)$ and every Borel set $E\subset \pi_VA$
with $\mu_V(E)=0$, we have
\[
 1=\int_{\pi_VA\setminus E}f_V\dd\LL_V^k
 \le M_0\LL_V^k(\pi_VA\setminus E)+\frac12.
\]
Thus
\begin{equation}\label{eq:carrier}
 \LL_V^k(\pi_VA\setminus E)\ge a_0:=\frac1{2M_0}>0.
\end{equation}

Fix $V,E$ as above. For each finite word $w$, use $V_w$ and
$S_{V,w}$ as defined in Section~\ref{subsec:conditional-cylinders}.
Note that we can always bound from below: $\mu\ge p_w(S_w)_*\mu$ by self-similarity. Therefore
$(S_{V,w})_*\mu_{V_w}\ll\mu_V$.
This implies
\[
 E_w=\{z\in \pi_{V_w}A:S_{V,w}(z)\in E\}
\]
 has $\mu_{V_w}$-measure zero. Use the Jacobian of $S_{V,w}$ then apply \eqref{eq:carrier} to $E_w$:
\[
 \LL_V^k\bigl((\pi_VS_wA)\setminus E\bigr)
 =r_w^k\LL_{V_w}^k(\pi_{V_w}A\setminus E_w)
 \ge a_0r_w^k.
\]
Thus Lemma~\ref{lem:density} implies $\LL_V^k(E)=0$. The reverse absolute continuity was already proved so the two measures are equivalent.
\end{proof}

Under the hypotheses of Theorem~\ref{thm:energy}, the uniform
density-tail bound implies the assumption of Proposition~\ref{prop:equivalence}. Then applying this proposition to the densities constructed in
Sections~\ref{sec:smoothing} and~\ref{sec:unequal} proves the equivalence
claim in Theorem~\ref{thm:energy}. It also implies the
Lebesgue-almost-everywhere conditional-measure statement under \SSC.

Proposition~\ref{prop:equivalence} implies \eqref{eq:carrier} as well. Taking $E=\varnothing$ there we obtain
\[
 \inf_{V\in\Gr(d,k)}\LL_V^k(\pi_VA)\ge a_0>0.
\]
This proves the uniform positive projection-volume lower bound and
completes the proof of Theorem~\ref{thm:energy}.

\section{Self-similar sets}\label{sec:generalsets}
We will use the natural self-similar measures defined on strongly separated subsystems to deduce
Theorem~\ref{thm:sets}.

\subsection{Strongly separated subsystems}
To prove Theorem \ref{thm:sets}, we state the following subsystem approximation
in \cite[Proposition~1.8]{Farkas16}.

\begin{proposition}\label{prop:subsystem}
Under the hypotheses of Theorem~\ref{thm:sets}, write $s=\dimH A$.
For every $0<\varepsilon<s-k$, there is an IFS
\[
 T_i(x)=\lambda_iR_ix+b_i,\qquad 1\le i\le m_\varepsilon,
\]
formed by finite compositions of the original maps, satisfying
\SSC\ and \DRC, and with attractor $A_\varepsilon\subset A$ such that
$s_\varepsilon:=\dimH A_\varepsilon>s-\varepsilon$.
\end{proposition}

For this subsystem, let $\mu_\varepsilon$ be the self-similar
measure with weights $p_i=\lambda_i^{s_\varepsilon}$.
By \cite[Theorem~9.3]{Falconer14}, one can find a constant $C_\varepsilon$ such that
\[
 \mu_\varepsilon(B(x,r))\le C_\varepsilon r^{s_\varepsilon}
 \quad(x\in\R^d,\ 0<r\le1).
\]
Thus $I_t(\mu_\varepsilon)<\infty$ for every $0<t<s_\varepsilon$. Moreover, $\dimH \mu_\varepsilon=s_\varepsilon$. Since $s_\varepsilon>k$, these measures satisfy the finite-energy
hypothesis of Theorem~\ref{thm:energy}, by choosing
$k<t_0<\min(s_\varepsilon,d)$. The absolute-continuity and
density estimates therefore apply to them.

\subsection{Transferring to the original set}
We will use the standard slicing upper bound: if $K\subset\R^d$
is compact and $\dimH K>k$, then, for every fixed $V\in\Gr(d,k)$,
\begin{equation}\label{eq:slicing-upper}
 \dimH(K\cap\pi_V^{-1}\{y\})\le\dimH K-k
 \quad\text{for }\LL_V^k\text{-almost every }y\in V.
\end{equation}
This is a consequence of Mattila
\cite[Theorem~7.7]{Mattila95}, applied to $\pi_V$.
The exceptional set may be enlarged to a Borel null set.

\begin{proof}[Proof of Theorem~\ref{thm:sets}]
Write $s=\dimH A$. Fix $0<\varepsilon<s-k$, and choose
$A_\varepsilon$ as in Proposition~\ref{prop:subsystem}, with
the self-similar measure $\mu_\varepsilon$.
Then $s_\varepsilon=\dimH \mu_\varepsilon=\dimH A_\varepsilon>s-\varepsilon>k$ and
$I_{t_0}(\mu_\varepsilon)<\infty$ for some $k<t_0<d$. Theorem~\ref{thm:energy} implies
\begin{equation}\label{eq:subvolume}
 a_\varepsilon:=
 \inf_{W\in\Gr(d,k)}\LL_W^k(\pi_WA_\varepsilon)>0,
\end{equation}
and the fibre measures of exact dimension $s_\varepsilon-k$ for the projected self-similar measures almost surely, which by the equivalence, can be moved to the restricted Lebesgue measures on the projected planes. The slicing bound \eqref{eq:slicing-upper} implies the
right upper bound. Thus these fibres have dimension
$s_\varepsilon-k$ almost everywhere.
The inclusion $A_\varepsilon\subset A$ immediately proves
the uniform positive-volume claim for $A$.

Fix $V\in\Gr(d,k)$. For a word $w$ in the original IFS,
use the notation $V_w$ and $S_{V,w}$ from
Section~\ref{subsec:conditional-cylinders}.
For the empty word, we use $S_{\varnothing}=\mathrm{id}$,
$O_{\varnothing}=I$, and $r_{\varnothing}=p_{\varnothing}=1$.
The identity $\pi_V S_w=S_{V,w}\pi_{V_w}$ shows that $S_{V,w}$
maps $V_w$ onto $V$ with volume factor $r_w^k$.
For every plane $W$ in this countable family indexed by the set of finite words, choose a Borel set
$G_W\subset\pi_WA_\varepsilon$ of full restricted
Lebesgue measure on which the fibre dimension statement holds.
This choice is possible because the exceptional sets in
Section~\ref{sec:conditional} and in \eqref{eq:slicing-upper}
can be taken Borel, and $\pi_WA_\varepsilon$ is compact. In particular,
$\LL_W^k(G_W)=\LL_W^k(\pi_WA_\varepsilon)\ge a_\varepsilon$.
Set
\[
 G_{\varepsilon,V}=\bigcup_w S_{V,w}(G_{V_w})\subset\pi_VA.
\]
Each $S_{V,w}$ is a homeomorphism between the corresponding
planes, so it takes Borel sets to Borel sets. There are only
countably many finite words, and hence $G_{\varepsilon,V}$
is Borel.
For $y=S_{V,w}(z)$ in this set, the fibre of $A$ at $y$
contains
\[
 S_w(A_\varepsilon\cap\pi_{V_w}^{-1}\{z\}),
\]
which has dimension $s_\varepsilon-k>s-k-\varepsilon$.
To show that $G_{\varepsilon,V}$ has full Lebesgue measure in $\pi_VA$,
let $E=\pi_VA\setminus G_{\varepsilon,V}$. For every word $w$,
$S_{V,w}(G_{V_w})\subset(\pi_VS_wA)\setminus E$, and hence
\[
 \LL_V^k\bigl((\pi_VS_wA)\setminus E\bigr)
 \ge r_w^k\LL_{V_w}^k(G_{V_w})
 \ge a_\varepsilon r_w^k.
\]
Then we may use again Lemma~\ref{lem:density} to deduce that $\LL_V^k(E)=0$.

Now take $\varepsilon_j=(s-k)2^{-j-1}$, $j\ge1$,
choosing the subsystem separately for each $j$.
The Borel set
\[
 G_V^*=\bigcap_{j\ge1}G_{\varepsilon_j,V}
\]
has full Lebesgue measure in $\pi_VA$, since its complement
consisting a countable union of null sets. At every point of
$G_V^*$, the fibre has dimension greater than
$s-k-\varepsilon_j$ for every $j$, and consequently
at least $s-k$.
Intersect $G_V^*$ with the full-measure set supplied by
\eqref{eq:slicing-upper} for the original compact set $A$.
The reverse inequality holds on this intersection, proving
the claimed equality.
The set of base points where equality holds has positive
$k$-dimensional Lebesgue measure and hence dimension $k$. Thus the set dimension conservation \eqref{eq:dc-set} holds.
\end{proof}

\section{Prospects and open questions}\label{sec:prospects}

We finish with a few questions about possible extensions of the $L^1$ smoothing method.

\emph{The planar set theorem.}
The counterexample by Rapaport \cite[Theorem~1.1]{Rapaport17} shows that the $d = 2$ version of Theorem \ref{thm:energy} is false, but it leaves open the corresponding analogue of Theorem \ref{thm:sets} for sets. That is, could $\pi_V \mu$ for $V \in \Gr(2,1)$ be singular even if its support has
positive Lebesgue measure? Suppose that $A\subset\R^2$ is self-similar,
has dense rotations, and $\dimH A=s>1$. Is the Lebesgue measure of $\pi_V A$ positive for all $V \in \Gr(2,1)$ with fibres of dimension $s-1$ at
Lebesgue-almost every point? Our use of
Varj\'u's estimate relies on the semisimple structure of $\SO(d)$ for
$d\geq 3$ and does not apply to $\SO(2)$, so a planar argument would
have to be different.

\emph{Self-affine projections and dimension conservation.}
Consider an IFS $x\mapsto A_i x+a_i$ with invertible contracting
linear parts. Under what assumptions on the matrices $A_i$, together
with finite $t$-energy for some $t>k$, are all $k$-dimensional
projections absolutely continuous, with uniformly integrable
densities? One can also ask when the conditional measures have
dimension $\dimH\mu-k$ for $\mu_V$-almost every base point, and when
the fibres of the attractor have dimension $\dimH A-k$ at
Lebesgue-almost every point of each projection.

The $L^1$ scaling itself still behaves well under an invertible
linear map. If $T:\R^k\to\R^k$ is invertible and
\[
h_T(x)=|\det T|^{-1}h(T^{-1}x),
\]
then $\|h_T\|_1=\|h\|_1$. The difficulty is that a Gaussian is no
longer taken to a Gaussian of the same shape: its covariance changes,
and the projection plane also changes under products of the matrices.
The relevant walk on the Grassmannian therefore has in general a
stationary measure different from $\sigma_k$ on $\Gr(d,k)$. To use our argument similarly one would need an averaged smoothing estimate with respect
to this stationary measure, together with some control of the changing
covariances. The Fourier decay results
\cite{LiSahlsten20,Solomyak22,FengYi26} and the projection results
\cite{FengXie25,MorrisSert26,BaranyKaenmakiKolossvary26} may provide
natural settings for such a question.

\emph{Stationary measures and their projections.}
One can ask a similar question for more general stationary measures
$\nu=\sum_i p_i(g_i)_*\nu$ of smooth maps $g_i$, including projective
actions of random matrices. Li~\cite{Li22} proved Fourier decay for
Furstenberg measures of Zariski-dense random walks on split semisimple
groups under moment assumptions. For absolute continuity,
Kittle~\cite{Kittle25} gave a sufficient criterion for finitely
supported walks on $\mathrm{PSL}_2(\R)$, while
Lequen~\cite{Lequen25} constructed finitely supported walks on simple
Lie groups whose Furstenberg measures are smooth. In the opposite
direction, B\'enard~\cite{Benard26} recently proved the singularity conjecture for
$\mathrm{PSL}_2(\R)$: if the support of a finitely supported measure
generates a Zariski-dense discrete subgroup, then its Furstenberg
measure is singular. It would be interesting to know whether the $L^1$ smoothing method developed here can be adapted to stationary measures to prove absolute continuity of the measure or of its projections.

\section*{Acknowledgements}
We are especially grateful to Simon Baker for useful comments, critiques, suggestions, and conversations concerning this work and an earlier version.

\bibliographystyle{amsplain}
\bibliography{references}
\end{document}